\documentclass[preprint, 1p,11pt]{amsart}

\usepackage{amssymb,amsmath,amsthm}
\usepackage{calc}
\usepackage{multicol}
\usepackage[shortlabels]{enumitem}
\usepackage{wrapfig}
\SetEnumitemKey{twocol}{
  itemsep=1\itemsep,
  parsep=1\parsep,
  before=\raggedcolumns\setlength{\columnsep}{3em}\begin{multicols}{2},
  after=\end{multicols}\vspace{-\parskip}
}
\setlist[enumerate]{format=\textnormal}
\newlist{itemproof}{itemize}{1}
\setlist[itemproof]{label=\textbullet, wide=0pt}

\usepackage[usenames]{color}
\usepackage[colorlinks]{hyperref}
\hypersetup{
colorlinks,
pdfborder = 1 0 1,
linkcolor = black,
anchorcolor = black,
citecolor = black,
filecolor = black,
menucolor = black,
urlcolor = black
}
\usepackage[initials, backrefs]{amsrefs}
\usepackage{tikz-cd}
\usepackage{xspace}
\usepackage{lineno}
\modulolinenumbers[5]

\theoremstyle{plain}
\newtheorem{theorem}{Theorem}[section]
\newtheorem{lemma}[theorem]{Lemma}
\newtheorem{corollary}[theorem]{Corollary}
\newtheorem{proposition}[theorem]{Proposition}
\theoremstyle{definition}
\newtheorem{definition}[theorem]{Definition}
\theoremstyle{remark}
\newtheorem{remark}[theorem]{Remark}

\newtheorem{example}[theorem]{Example}
\newtheorem{thm}{\bfseries\upshape Theorem}

\newtheorem{cor}[thm]{\bfseries\upshape Corollary}
\newtheorem{lem}[thm]{\bfseries\upshape Lemma}

\newcommand{\m}[1]{{\mathbf{#1}}}   
\newcommand\cls[1]{{\mathcal{#1}}}   

\newcommand{\Rl}{\ensuremath{\mathcal{RL}}\xspace}

\newcommand{\N}{{\mathbb N}}       
\newcommand{\Z}{{\mathbb Z}}        

\renewcommand\iff{\Leftrightarrow}                   
\newcommand\da{\mathop{\downarrow}}          
\newcommand\ua{\mathop{\uparrow}}               
\newcommand{\Ra}{\Rightarrow}                       
\newcommand\rest{\mathnormal{\upharpoonright}}           

\newcommand\Con{\operatorname{Con}}              
\newcommand\cg[1]{\operatorname{Cg}_{\scriptscriptstyle#1}}   
\newcommand\Crow{\operatorname{{\mathcal C}^{\infty}}}   
\newcommand\Diag{\operatorname{Diag}}             
\newcommand\dom{\operatorname{dom}}              
\newcommand\End{\operatorname{End}}                
\newcommand\F{\operatorname{\mathbf F}}           
\renewcommand\Join[1][\empty]{\bigvee\nolimits^{#1}}       
\newcommand\Low{\operatorname{{\mathcal L}}}                 
\newcommand\McN[1]{\if#1*\operatorname{{\mathcal N}*}\else\operatorname{{\mathcal N}}#1\fi}  
\newcommand\Meet[1][\empty]{\bigwedge\nolimits^{\!#1}}  
\let\wp\power							 

\newcommand*\conj{\mathbin{\;\&\;}}   
\newcommand{\jn}{\lor}                        
\newcommand{\mt}{\land}                     
\newcommand\ut{{\mathrm{e}}}            
\newcommand{\ld}{\backslash}                           
\newcommand{\rd}{/}                                          
\newcommand{\ras}{\rightsquigarrow}
\newcommand{\rds}[1]{\rd_{_{\!\!#1}}}                
\newcommand{\lds}[1]{\ld_{_{#1}}}                                            
\newcommand{\biglds}[1]{\big\ld_{\scriptscriptstyle#1}}          
\newcommand{\Biglds}[1]{\Big\ld_{\scriptscriptstyle#1}}          
\newcommand{\tos}[1]{\rightarrow_{\scriptscriptstyle#1}}  

\newcommand\ds[1]{_{_{#1}}}			     
\newcommand\eq{\approx}                                 
\renewcommand{\L}{\ensuremath{\mathcal{L}}}   
\newcommand{\ol}{\overline}                               
\newcommand\pair[1]{\langle #1\rangle}             
\newcommand\textss{\textsuperscript}                
\newcommand{\X}{{\mathbb{X}}}                          
\newcommand{\lan}{\langle}                               
\newcommand{\ran}{\rangle}                              
\newcommand\leqc{\preccurlyeq}		    
\newcommand\nleqc{\not\leqc}			    
\newcommand\nbd[1]{\protect\nobreakdash-\hspace{0pt}}		

\definecolor{grey}{rgb}{0.6,0.6,0.6}

\newcounter{commento}

\begin{document}

\title{Join-Completions of Ordered Algebras}

\author{Jos\'e Gil-F\'erez}
\address{
University of Bern, Switzerland}
\email{jose.gil-ferez@math.unibe.ch}

\author{Luca Spada}
\address{
University of Salerno, Italy}
\email{lspada@unisa.it}

\author{Constantine Tsinakis}
\address{
Vanderbilt University, U.S.A}
\email{constantine.tsinakis@vanderbilt.edu}

\author{Hongjun Zhou}
\address[xian]{
Shaanxi Normal University,
China}
\email{sdzhjun@gmail.com}


\maketitle

\begin{abstract}
We present a systematic study of join-extensions and join-completions of ordered algebras, which naturally leads to a refined and simplified treatment of  fundamental results and constructions in the theory of ordered structures ranging from properties of the Dedekind-MacNeille completion to the proof of the finite embeddability property for a number of varieties of ordered algebras.
\end{abstract}

\section{Introduction}\label{s:intro}
This work presents a systematic study of join-extensions and join-com\-ple\-tions of ordered algebras, which provides a uniform and refined treatment of  fundamental results and constructions ranging from properties of the Dedekind-MacNeille completion to the proof of the finite embeddability property for a number of varieties of ordered algebras.

Given two ordered algebras $\m P$ and $\m L$ of the same signature, we say that $\m{L}$ is a \emph{join-extension} of $\m P$ or that $\m P$ is \emph{join-dense} in $\m{L}$ if the order of $\m  L$ restricts to that of $\m P$ and, moreover, every element of $\m{L}$ is a join of elements of $\m P$. The term \emph{join-completion} is used for a join-extension whose partial order is a complete lattice. By an ordered algebra we understand a structure in the sense of model theory in which one of the relations is a partial order. In all cases the structures contain one or two monoidal operations that are compatible with, or even residuated with respect to, the partial order.   In general, we do not assume that the algebra reduct of $\m P$ is a subalgebra of that of $\m L$. The concepts of a \emph{meet-extension} and a \emph{meet-completion} are defined dually. 

Here is a summary of the contents of the article. 
In Section \ref{section:preliminaries}, we dispatch some preliminaries on partially ordered monoids, residuated partially ordered monoids, residuated lattices, nucleus-systems and nuclei.  Section~\ref{section<<joinextensions} explores the following question: Under what conditions a join-completion of a partially ordered monoid $\m P$ is a residuated lattice with respect to a (necessarily unique) multiplication that extends the multiplication of $\m P$?   The answer is provided by Theorem A below. {Before stating the theorem, we note that for a given partially ordered monoid $\m P$, there is a unique up to isomorphism largest join-completion $\Low(\m P)$ of $\m P$ whose multiplication is residuated and extends the multiplication of $\m P$. We use the same symbol $\Low(\m P)$ to denote the structure with the residuals of the multiplication added. (See Section~\ref{section:preliminaries} for details.)}

\begin{thm}[See Theorem~\ref{thm:completions-mainresult}.]
\emph{Let $\m P=\pair{P,\leq,\cdot,\ut}$ be a partially ordered monoid and let $\m L$ be a join-completion of the partially ordered set $\pair{P,\leq}$. The following statements are equivalent:
\begin{enumerate}[(i)]
\item $\m{L}$ can be given a structure of a residuated lattice whose multiplication extends the multiplication of $\m P$. 
\item For all $a \in P$ and $b \in L$, the residuals $a \lds{\Low(\m P)} b$ and $b \rds{\Low(\m P)} a$ are in $L$. 
\item $L$ is a nucleus-system of $\Low(\m P)$. 
\item The closure operator $\gamma_L$ associated with $L$ is a nucleus on $\Low(\m P)$. 
\end{enumerate}
Furthermore, whenever the preceding conditions are satisfied, the multiplication on $\m L$ is uniquely determined and the inclusion map $\m{P}\hookrightarrow \m{L}$ preserves, in addition to multiplication, all existing residuals and meets.}
\end{thm}

The preceding theorem provides a simple proof of the fact that the De\-de\-kind-MacNeille completion of a residuated partially ordered monoid is a residuated lattice and also of the fact that the Dedekind-MacNeille completion of an implicative semilattice is a Heyting algebra. More importantly, it implies the following result which will play a key role in the proofs of the finite embeddability results of Section~\ref{sec:FEP:HRL}.  

\begin{cor}[See Corollary~\ref{completions-RL-Heyting}.]
\emph{Let $\m P$ be an integral meet-semilattice monoid and $\m{L}$ a join-completion of $\pair{P,\leq}$. Then, the following statements are equivalent:
\begin{enumerate}[(i)]
\item {$\m{L}$ can be given both the structure of a residuated lattice whose multiplication extends the multiplication of $\m P$ and of a Heyting algebra with respect to the lattice reduct of $\m{L}$.}
\item For all $a \in P$ and $b \in L$, the residuals $a \lds{\Low(\m P)} b$, $b \rds{\Low(\m P)} a$ and the Heyting implication $a \tos{\Low(\m P)} b$ are in $L$. 
\item $L$ is a nucleus-system of the algebras $ \pair{\Low(\m P),\mt,\jn,\cdot,\lds{\Low(\m P)},\rds{\Low(\m P)}, \ut}$ and \break $ \pair{\Low(\m P),\mt,\jn,\mt,\tos{\Low(\m P)},\ut}$. 
\item The closure operator $\gamma_L$ associated with $L$ is a nucleus on   the algebras $\pair{\Low(\m P),\mt,\jn,\cdot,\lds{\Low(\m P)},\rds{\Low(\m P)}, \ut}$ and $ \pair{\Low(\m P),\mt,\jn,\mt,\tos{\Low(\m P)},\ut}$. 
\end{enumerate}
Furthermore, whenever the preceding conditions are satisfied, the two structures are uniquely determined and the inclusion map $\m{P}\hookrightarrow \m{L}$ preserves multiplication, all existing residuals (including Heyting implication) and meets.}
\end{cor}

Section~\ref{s:DM-completion} explores join-completions of involutive residuated partially ordered monoids, in particular, involutive residuated lattices. It is convenient to think of involutive residuated partially ordered monoid $\m P$ as residuated partially ordered monoids endowed with a cyclic dualizing element $d$. A \emph{cyclic} element $d\in P$ is one satisfying  $d\rd x = x\ld d$, for all $x\in P$. Denoting the common value $d\rd x = x\ld d$ by $x \ras d$, a \emph{cyclic dualizing} element is a cyclic element $d$ satisfying $(x\ras d)\ras d=x$, for all $x\in P$. It is straightforward to show that  the map $\gamma_d\colon x\mapsto (x \ras d)\ras d$ is a nucleus whenever $d$ is a cyclic element. The following result shows how cyclic elements give rise to involutive residuated partially ordered monoids.

\begin{lem}[See Lemma~\ref{lem:involuted1}.]
Let $\m P$ be a residuated partially ordered monoid and $\gamma$ a nucleus on $\m P$. Then the nucleus-system $\m P_\gamma$ is an involutive residuated partially ordered monoid if and only if there exists a cyclic element $d$ of $\m P$ such that $\gamma = \gamma_d$. 
\end{lem}

The next result generalizes Theorem~4.3 of~\cite{Sch77} and may be viewed as a natural extension of the Glivenko-Stone Theorem (\cite{Gli29},~\cite{Sto36}), which states that the  Dedekind-MacNeille  completion of a Boolean algebra is a Boolean algebra. More specifically, we have:

\begin{thm}[See Theorem~\ref{thm:involuted2}.]
Let $\m P$ be a residuated partially ordered monoid and let $\m{L}$ be  a join-completion of $\m P$ which is a residuated lattice with  respect to a multiplication that extends the multiplication  of $\m P$. Then for every cyclic dualizing element $d \in P$,  $\m{L}_{\gamma_d}$ is the Dedekind-MacNeille completion of $\m{P}_{\gamma_d}$. 
\end{thm}

An interesting application of the preceding result is a  succinct and com\-pu\-ta\-tion-free proof of the fact that the Dedekind-MacNeile completion of an Archimedean partially ordered group is a conditionally complete partially ordered group (See Theorem~\ref{thm:integrallyclosed}).

In Section~\ref{sec:FEP:HRL}, we make use of the results of Section~\ref{section<<joinextensions} to produce refined algebraic proofs of existing and new results on the finite embeddability property (FEP). The standard process of establishing this property for a variety of ordered algebras usually consists of producing a ``potentially'' finite extension of a finite partial algebra and then proving that this extension is finite. In our approach, the theory of join-extensions is employed in the construction of the extension, while a modification of the fundamental ideas of Blok and van Alten in~\cite{BvA05} establishes its finiteness. Our approach is illustrated in the proof of the FEP for  the variety $\cls{HRL}$ of Heyting residuated lattices (see Lemma~\ref{lem:embedding}), but it easily applies to the results in ~\cite{BvA05}. This variety consists of all algebras $\m A =\pair{A,\mt,\jn,\cdot,\ld,\rd,\to,\ut}$ that combine compatible structures of a residuated lattice and a Heyting algebra on the same underlying lattice. Even though the variety $\cls{HRL}$ has not received much attention in the literature, the introduction of the Heyting arrow guarantees that the construction maintains lattice-distributivity. In particular, it implies the FEP for the variety of distributive integral  residuated lattices, a result that has been obtained independently in~\cite{Bus11} and~\cite{GaJi} by alternative means. 
Thus we have:

\begin{thm}[See Theorem~\ref{thm:HRL:FEP}.]\label{hrl}
The variety $\cls{HRL}$ of Heyting residuated lattices has the finite embeddability property.
\end{thm}

\begin{cor}[See Corollary~\ref{DIRL:FEP}.]
The variety of distributive integral residuated lattices has the finite embeddability property.
\end{cor}

Further, combining the results of Section~\ref{s:DM-completion} with the approach used in the proof of Theorem~\ref{hrl}, we have:

\begin{thm}[See Theorem~\ref{thm:INVIRL:FEP}.]
The variety $\cls{I}nv\cls{IRL}$ of involutive integral residuated lattices has the finite embeddability property.
\end{thm}

Lastly, the aim of Section~\ref{s:fep} is to provide a survey of the finite embeddability property by clarifying relationships among several related notions -- such as finitely presented algebras, finite model property, residual finiteness, and the word problem -- and reviewing general theorems with detailed proofs that remedy some gaps in the literature of ordered structures. It appears to us that there is no reference in the literature of ordered algebras where these interrelationships are discussed in detail, in particular how the partial order of such a structure affects the notion of finite embeddability. 

\section{Preliminaries}\label{section:preliminaries}

In this section we review the notions of a partially ordered monoid, residuated partially ordered monoid, residuated lattice, nucleus-system and nucleus. These concepts and their properties will play a key role in the remainder of this article.

Let $\m P=\lan P, \leq \ran$ and $\m P'=\lan P', \leq' \ran$ be partially ordered sets. A map $\varphi \colon P \to P'$ is said to be an \emph{order-homomorphism}, or \emph{order-preserving}, if for all $p,q\in P$, $p\leq q$ implies $f(p)\leq'f(q)$; an \emph{order-embedding} if for all $p,q\in P$, $p\leq q$ if and only if $f(p)\leq'f(q)$; and an \emph{order-isomorphism} if it is bijective and an order-embedding. A subset $X\subseteq P$, the \emph{lower set} of $X$ is the subset $\da X = \{p\in P \mid p\leq x, \text{ for some } x\in X\}$ of $P$; dually, the \emph{upper set} of $X$ is the subset $\ua X = \{p\in P \mid x\leq p, \text{ for some } x\in X\}$ of $P$. In what follows,  we use the abbreviations $\da a$ for $\da \{a\}$ and $\ua a$ for $\ua \{a\}$, whenever $a\in P$. An \emph{order-ideal} of $\m P$ is a subset $I$ of $P$ satisfying $I = \da I$. A~\emph{principal} order-ideal is one of the form $\da a$, for some $a\in P$. \emph{Order-filters} and \emph{principal} order-filters are defined dually.

A \emph{closure operator} on $\m P$ is a map $\gamma \colon P\to P$ with the usual properties of being an order-homomorphism (that is, order-preserving), enlarging ($x \leq \gamma(x)$), and idempotent ($\gamma(x) = \gamma(\gamma(x))$). It is completely determined by its image
\begin{equation}\label{eq:1.1}
  P_\gamma = \{\gamma(x) \mid x\in P\} = \{x\in P \mid x = \gamma(x)\},
\end{equation} 
by virtue of the formula
\begin{equation}\label{eq:1.2} 
\gamma(x) = \min \{ p \in P_\gamma \mid x \leq p\}.
\end{equation}
 
 A \emph{closure system} of $\m P$ is a subset $C\subseteq P$ such that for all $x\in P$, $\min\{ p \in C \mid x \leq p\}$ exists. Conditions \eqref{eq:1.1} and \eqref{eq:1.2} establish a bijective correspondence between closure operators on and closure systems of $\m P$. In what follows, we use $\gamma_C$ to denote the closure operator associated to a closure system $C$. Every closure system $C$ of $\m P$ inherits from $\m P$ the structure of a partially ordered set $\m C$. It can be readily seen that if $X\subseteq C$ is such that $\Join[\m P] X$ exists, then $\Join[\m C] X$ exists and $\Join[\m C] X = \gamma_C\big(\Join[\m P] X \big)$. Also, if $X\subseteq C$ is such that $\Meet[\m P] X$ exists, then $\Meet[\m C] X$ exists and $\Meet[\m C] X = \Meet[\m P] X$. In particular, the closure systems of a complete lattice $\m P$ are the nonempty subsets of $P$ that are closed with respect to arbitrary meets in $\m P$ of their elements. In this case, for each such $C$, the poset $\m C$ is also a complete lattice in which arbitrary meets, but not joins in general, are preserved in $\m P$.

Any partially ordered set gives rise to a concrete situation of  the concepts described in the preceding paragraph. Consider any partially ordered set $\m P$ and let  $\wp(P)=\pair{\wp(P),\subseteq}$ be the partially ordered set of all subsets of $P$ under set\nbd-inclusion and let $\Low(\m P)=\pair{\Low(P),\subseteq}$ be the partially ordered set of all order-ideals of $\m P$. The latter two are complete lattices in which arbitrary joins and meets are just unions and intersections, respectively. Further, the map $\gamma_{\da} \colon \wp(P)\to\wp(P)$, defined by $\gamma_{\da} (X) = \da X$, is a closure operator on $\wp(P)$ whose associated closure system is  $\Low(P)$.

A \emph{partially ordered monoid}, or \emph{pomonoid}, is a structure $\m{P} = \pair{P,\leq, \cdot, \ut}$ consisting of a partial order and a monoidal structure such that the product is compatible with the order, meaning that the product is order-preserving in both coordinates. As is customary, we use juxtaposition $xy$ instead of $x\cdot y$, when there is no danger of confusion. A partially ordered monoid is called \emph{integral} if the identity of the monoid is also the top element of the order.

Given a partially ordered monoid $\m P$ and two elements $a, b\in P$, the \emph{left} and \emph{right residuals} of $b$ by $a$, if they exist, are the elements
\begin{equation}\label{d:rsiduals}
a\ld b = \max\{x \in P \mid ax\leq b\}  \quad\text{and}\quad  b\rd a  = \max\{x \in P \mid xa\leq b\}.
\end{equation}
Thus, if the left residual of $b$ by $a$ exists, then for every $x\in P$, $ax\leq b$ if and only if $x\leq a\ld b$, and analogously for the right residual. A \emph{residuated partially ordered monoid} is a partially ordered monoid in which all residuals exist. We will view it as a structure $\pair{P,\leq,\cdot,\ld,\rd,\ut}$ satisfying the equivalences 
\begin{equation}\label{d:rsiduals-eq}
xy\leq z \quad\iff\quad y\leq x\ld z \quad\iff\quad x\leq z\rd y,
\end{equation}
for all $x, y, z\in P$, and refer to the two operations $\ld$ and $\rd$ as the \emph{left residual} and \emph{right residual} of multiplication. Finally, a \emph{residuated lattice} is a structure $\pair{P,\jn,\mt,\cdot, \ld, \rd, \ut}$, which is both a lattice and a residuated partially ordered monoid with respect to the induced order. 

The class $\Rl$ of residuated lattices is a finitely based variety.  The defining equations of $\Rl$ consist of the defining  equations for lattices and monoids together with the equations below.
\begin{enumerate}[(RL1), twocol, noitemsep, align=right, leftmargin=3\parindent]
 \item $x(y\jn z) \eq xy\jn xz$
 \item $(y\jn z)x \eq yx\jn zx$
 \item $x\ld y\leq x\ld (y\jn z)$
 \item $y\rd x\leq  (y\jn z)\rd x$
 \item $x(x\ld y) \leq y \leq x\ld xy$
 \item $(y\rd x)x\leq y \leq yx\rd x$
\end{enumerate}

Next lemma is a well-known result, and we will use it in what follows without an explicit mention to it.

\begin{lemma}\label{lem:residuation:joins:meets}
If $\pair{P, \leq, \cdot, \ld,\rd, \ut}$ is a residuated partially ordered monoid, then
\begin{enumerate}[(i)]
 \item the product preserves all existing joins in each argument; and
 \item the residuals preserve all existing meets in the numerator and convert all existing joins in the denominator into meets.
\end{enumerate}
\end{lemma}

A \emph{nucleus} on a partially ordered monoid $\m P$ is a closure operator $\gamma$ on $\m P$ satisfying the inequality 
\[
\gamma (a) \gamma(b) \leq \gamma(ab),
\]
for all $a,b \in P$. A \emph{nucleus-system} of $\m P$ is a closure system $C$ of $\m P$ 
satisfying 
\begin{equation}\label{eq:nuclear}
  x \ld a \in C \text{ and } a \rd x \in C,  \text{ for all } x\in P \text{ and } a\in C.
\end{equation}

The next result describes the relationship between nuclei and nucleus-systems (see~\cite{Ros90b}*{p.~31} or~\cite{GT05}*{Lemma 3.1}, and~\cite{Sch77}*{Corollary~3.7} for an earlier result in the setting of Brouwerian meet-semilattices). 

\begin{lemma}\label{lem:nucleus} 
Let $\gamma$ be a closure operator on a residuated pomonoid ${\m P}$, and let $P_{\gamma}$ be the closure system associated with $\gamma$. The following statements are equivalent:
\begin{enumerate}[(i)]
\item $\gamma$ is a nucleus.

\item $\gamma(a) \rd b,\;b\ld\gamma(a)\in P_{\gamma}$ for all $a,b\in P$.

\item $P_{\gamma}$ is a nucleus-system of ${\m P}$.

\item $\gamma(a) \rd b=\gamma(a) \rd \gamma(b)$ and $b\ld\gamma(a)=\gamma(b)\ld\gamma(a)$ for all $a,b\in P$.
\end{enumerate}
In particular, equations~(\ref{eq:1.1}) and~(\ref{eq:1.2}) establish a bijective correspondence between nuclei and nucleus-systems of $\m P$.
\end{lemma}
 
The next two lemmas, proved in~\cite{GT05}*{Lemma 3.3}, show that a nucleus-system $C$ of a residuated partially ordered  monoid $\m P$ inherits the structure of a residuated partially ordered monoid $\m C = \pair{C,\leq,\circ_{_\m C},\lds{\m C},\rds{\m C},\gamma_C(\ut)}$, where for any $x, y \in C$
\[
x \circ_{_\m C} y = \gamma_C (x \cdot y),\quad 
x \rds{\m C} y= x \rd y, \quad\text{and}\quad x \lds{\m C} y = x \ld y.
\]
Further, $\gamma_C \colon \m{P} \to {\m C}$ is an order and monoid homomorphism. However, $\gamma_C$ need not preserve the residuals in general. If, in addition, $\m P$ is a residuated lattice, then $\m C$ is a residuated lattice with respect to the operations defined above and
\[
x \jn_{_\m C} y = \gamma_C (x \jn y) \quad\text{and}\quad    x \mt_{_\m C} y  = x \mt y.
\]

\begin{lemma}
Given a nucleus-system $C$ of a residuated partially ordered monoid $\m P$, the structure ${\m C} = \pair{C,\leq, \circ_{_\m C},\ld,\rd,\gamma_C(\ut)}$ is a residuated partially ordered monoid. Furthermore, $\gamma_{C} \colon \m{P} \to {\m C}$ is an order and monoid homomorphism.
\end{lemma}

\begin{lemma}\label{lem:nuclear:system}
Let $C$ be a nucleus-system of a residuated lattice $\m{L} = \lan L, \mt, \jn ,$ $\cdot ,\ld , \rd, \ut\ran$. Then the structure ${\m C} = \lan C, \mt, \jn_{_\m C} ,\circ_{_\m C}, \ld , \rd, \gamma_C(\ut) \ran$ is a residuated lattice.
\end{lemma}

The following example will be used in a number of occasions.

\begin{example}\label{ex:L(P):RL}
If $\m P$ is a partially ordered monoid, then $\wp(\m{P}) = \lan \wp(P), \cap, \cup ,\cdot ,$ $\ld , \rd, \{\ut\}\ran $ is a residuated lattice, where: 
\begin{align*}
X\cdot Y &= \{ xy \mid x\in X, y\in Y\} \\ 
X\ld Y &= \{z \mid xz \in Y,\ \forall x \in X\} \\ 
Y\rd X &= \{z \mid zx \in Y,\ \forall x \in X\} 
\end{align*}
It is a simple matter to verify that $\gamma_{\da}$ is a nucleus on $\wp(\m{P})$, and hence, in light of Lemma~\ref{lem:nuclear:system}, $\Low(\m P)$ \label{definition:L(P)} has an induced structure of residuated lattice, which we call \emph{canonical}. The product of two order-ideals $X,Y\in \Low(\m P)$ is given by $X\circ_{\da} Y = \da (X\cdot Y)$. We also note, for future reference, that the product of two principal order-ideals $\da x$ and $\da y$ in $\Low(\m P)$ is the principal order-ideal $\da (xy)$. That is, $\da x \circ_{\da} \da y = \da (xy)$. If we think of $\m P$ as the partially ordered monoid  of principal order-ideals of $\m P$, the preceding observation states that the multiplication of $\Low(\m P)$ extends the multiplication of $\m P$.
\end{example}

\section{Join-extensions and join-completions of ordered algebras}
\label{section<<joinextensions}

The main result in this section is Theorem~\ref{thm:completions-mainresult}. It provides a description of any join-completion of a partially ordered monoid $\m P$ that is a residuated lattice with respect to a (necessarily unique) multiplication that extends the multiplication of $\m P$. Its proof will be preceded by the definition of relevant notions and proofs of auxiliary results.

Recall that a partially ordered set $\m{L}$ is said to be a \emph{join-extension} of a partially ordered set $\m P$, or that  $\m P$ is \emph{join-dense} in $\m{L}$, provided that $P$ is a subset of $L$, the order of $\m  L$ restricts to that of $\m P$, and every element of $\m{L}$ is a join of elements of $\m P$. A join-extension $\m L$ is a \emph{join-completion} if, in addition, $\m L$ is a complete lattice. The concepts of a \emph{meet-extension} and a \emph{meet-completion} are defined dually. Join-completions of partially ordered sets were introduced by B. Banaschewski~\cite{Ban56}, and were studied extensively by J.~Schmidt~\cites{Sch72a, Sch72b, Sch74}. They are intimately related to representations of complete lattices studied by J.R.~B\"uchi~\cite{Buc52}.

\begin{lemma}\label{lem:meet:faithful} 
 A join-extension $\m L$ of a partially ordered set $\m P$ preserves all existing meets in $\m P$. That is,  if $X \subseteq P$ and $\Meet[\m{P}] X$ exists, then $\Meet[\m{L}] X$ exists and $\Meet[\m{P}] X = \Meet[\m{L}] X$. Dually, a meet-extension of $\m P$ preserves all existing joins in $\m P$.
\end{lemma}
 
\begin{proof}
We prove the statement for join-extensions. Let $X \subseteq P$ be such that $\Meet[\m{P}] X$ exists. We need to prove that $\Meet[\m{L}] X$ exists and $\Meet[\m{P}] X = \Meet[\m{L}] X$.  Set $a=\Meet[\m{P}] X$ and let $b\in L$ be a lower bound of $X$. As $\m{L}$ is a join-extension of $\m P$, there exists $Y\subseteq P$, such that $b=\Join[\m{L}] Y$. Since $b$ is a lower bound of $X$ and at the same time an upper bound for $Y$, each $y\in Y$ is also a lower bound of $X$ in $\m P$. But $a$ is the greatest lower bound of $X$ in $\m P$, therefore $y\leq a$. Thus, $a$ is an upper bound for $Y$ and $b$ is the least such bound, hence $b\leq a$.  As $b$ is an arbitrary lower bound of $X$ in $\m{L}$, the latter inequality means that $a$ is the greatest lower bound of $X$ in $\m{L}$. We have shown that $\Meet[\m{L}] X =a = \Meet[\m{P}] X$.
 \end{proof}
 
\begin{proposition}\label{prop:closure-operators-systems}
Let $\m P$ be a partially ordered set, let $\m K$ be a join-completion of $\m P$, and let $L$ be a subset of $K$ that contains $P$. The partially ordered set $\m L$, with respect to the induced partial order from $\m K$, is a join-completion of $\m P$ if and only if it is a closure system of $\m K$. 
\end{proposition}

\begin{proof}
If $\m L$ is a join-completion of $\m P$, then $\m K$ is a join-completion of $\m L$ and, in view of Lemma~\ref{lem:meet:faithful}, arbitrary meets in $\m L$ are preserved in $\m K$. As moreover $\m L$ is complete, $L$ is a closure system of $\m K$.
Conversely, if $\m L$ is a closure system of $\m K$, then for every $X\subseteq L$, $\Join[\m L] X = \gamma_L\big(\Join[\m K] X\big)$. Therefore, if $a\in L\subseteq K$, then there exists $X\subseteq P$ such that $a = \Join[\m K]X$, whence $a = \gamma_L(a) = \gamma_L\big(\Join[\m K]X\big) = \Join[\m L]X$. Therefore, $\m L$ is a join-completion of $\m P$.
\end{proof}

If we take into account the order-isomorphism of $\m P$ with the partially ordered set $\dot{\m{P}} = \pair{\{ \da x \mid x \in P \},\subseteq}$ of its principal order-ideals, we see that $\Low(\m P)$ is a join-completion of $\m P$. Moreover, each join-extension of $\m P$ is isomorphic to a subpartially ordered set of $\Low(\m P)$. Indeed, if $\m L$ is a join-extension, then it is order-isomorphic to its \emph{canonical image} $\dot{\m L} = \pair{\{ \da x \cap P \mid x \in L \},\subseteq}$. Thus $\Low(\m P)$ is, up to isomorphism, the largest join-completion of $\m P$. It follows that every join-extension of $\m P$ can be embedded into $\Low(\m P)$ by an order-embedding that fixes the elements of $\m P$. Considering one realization of $\Low(\m P)$, the set of all intermediate partially ordered sets $\m P\subseteq \m L\subseteq \Low(\m P)$ contains an isomorphic copy of every join-extension of $\m P$. Therefore, the following result is an immediate consequence of Proposition~\ref{prop:closure-operators-systems}.

\begin{proposition}\label{proposition<<completions}
The join-completions of a partially ordered set $\m P$ are, up to isomorphism, the closure systems of $\Low(\m P)$ containing $P$, with the induced order.
\end{proposition}
 
Although it is possible to describe all join-extensions of a partially ordered set $\m P$ as systems of order-ideals of $\m P$, it is often more convenient to use abstract descriptions of them\,---\,after all, as noted in~\cite{Sch72a}, it would be cumbersome to always view the reals as Dedekind cuts of the rationals. For example, $\Low(\m P)$\,---\,the largest join-completion of $\m P$\,---\,can be described abstractly as the unique algebraic and dually algebraic distributive lattice whose partially ordered set of completely join-prime elements is isomorphic to $\m P$. The smallest join-completion of $\m P$, the so-called \emph{Dedekind-MacNeille completion} $\McN(\m{P})$, has an equally satisfying abstract description due to Banaschewski~\cite{Ban56}: it is the only join- and meet-completion of $\m P$. This characterization is a direct consequence of the fact that the canonical image of $\McN(\m{P})$ consists of all intersections of principal order-ideals of $\m P$. Observe that, in light of Lemma~\ref{lem:meet:faithful}, any existing meets and joins in $\m P$ are preserved in $\McN(\m{P})$. However, the so called \emph{Crawley completion} $\Crow(P)$\,---\,consisting of all order-ideals that are closed with respect to any existing joins of their elements\,---\,is the largest join-completion with this property (see~\cite{Sch72a, Sch72b}). In general, the inclusion $\McN(\m{P}) \subseteq \Crow(\m{P})$ is proper.   

Most of the implications of Theorem~\ref{thm:completions-mainresult} follow directly from the connection between nuclei and nucleus-systems (Lemma~\ref{lem:nucleus}). A crucial ingredient of the proof is the fact that residuals are preserved as one moves up the ladder of join-extensions. More specifically, we have the following auxiliary result.

\begin{lemma}\label{residuals}
Let $\m P$ be a partially ordered monoid and let $\m{L}$ be a join-extension of $\m P$ which, in addition, is a partially ordered monoid with respect to a multiplication extending the multiplication of $\m P$. Then for all $a, b \in P$, if $a \lds{\m{P}} b$ exists (in $\m P$), then $a \lds{\m{L}} b$ exists (in $\m{L}$) and
\[
a \lds{\m{P}} b = a \lds{\m{L}} b = a \lds{\Low(\m P)} b = a \lds{\wp(\m{P})} b.
\]
Likewise for the other residual.
\end{lemma}
 
\begin{proof}  Let $a, b \in P$ and assume that $a \lds{\m{P}} b$ exists. First, notice that $a\lds{\m P} b \in \{x\in L\mid ax\leq b\}$. We prove that $a \lds{\m{P}} b=\max\{x\in L \mid ax\leq b\}$, which clearly implies that $a \lds{\m{P}} b=a \lds{\m{L}} b$. To this end, let $x\in L$ such that $ax\leq b$. Since $\m{L}$ is a join-extension of $\m P$, there exists a subset $X\subseteq P$ such that $x = \Join[\m{L}]X$. Now, $ax \leq b$  implies that for every $p\in X$, $ap \leq a\Join[\m{L}]X\leq b$, and hence  $p \leq a \lds{\m P} b$. It follows that  $x = \Join[\m{L}]X \leq a \lds{\m P} b$, as was to be  shown.  Thus, indeed $a \lds{\m{P}} b = a \lds{\m{L}} b$, and also $a \lds{\m{P}} b = a \lds{\Low(\m P)} b$, since $\Low(\m P)$ is also a join-completion of $\m P$ whose multiplication extends that of $\m P$.

To complete the proof of the lemma, it will suffice to prove that $a \lds{\Low(\m P)} b = a \lds{\wp(\m{P})} b$. But this is an immediate consequence of Lemma~\ref{lem:nuclear:system}, as $\Low(\m P)$ is a nucleus-system of $\wp(\m{P})$ (see Example~\ref{ex:L(P):RL}).
\end{proof} 

\begin{theorem}\label{thm:completions-mainresult}
Let $\m P=\pair{P,\leq,\cdot,\ut}$ be a partially ordered monoid and let $\m L$ be a join-completion of the partially ordered set $\pair{P,\leq}$. The following statements are equivalent:
 \begin{enumerate}[(i)]
 \item\label{mainresult-item1}  $\m{L}$ can be given a structure of a residuated lattice whose multiplication extends the multiplication of $\m P$. 
\item\label{mainresult-item2} For all $a \in P$ and $b \in L$, $a \lds{\Low(\m P)} b \in L$ and $b \rds{\Low(\m P)} a \in L$. 
 \item\label{mainresult-item3} $L$ is a nucleus-system of $\Low(\m P)$. 
\item\label{mainresult-item4} $\gamma_L$ is a nucleus on $\Low(\m P)$. 
\end{enumerate}
Furthermore, whenever the preceding conditions are satisfied, the multiplication on $\m L$ is uniquely determined and the inclusion map $\m{P}\hookrightarrow \m{L}$ preserves, in addition to multiplication, all existing residuals and meets.
\end{theorem}

\begin{proof}
We first prove the equivalences. In light of Lemma~\ref{lem:nucleus}, (iii) and (iv) are equivalent. Thus, it will suffice to establish the  implications (i)${}\Ra{}$(ii)${}\Ra{}$(iii)${}\Ra{}$(i).
\begin{itemproof}
\item[(i)${}\Ra{}$(ii):] Suppose $\m{L}$ satisfies (i). Let  $a\in P$ and $b\in L$.  Since $\m L$ is a join-completion of $\m P$ whose multiplication extends that of $\m P$, we have that $\Low(\m P)$ is a join-extension of $\m L$ whose multiplication extends that of $\m L$, and hence Lemma~\ref{residuals} implies that  $a \lds{\Low(\m P)} b = a \lds{\m{L}} b$ and  $b \rds{\Low(\m P)} a = b \rds{\m{L}} a$. Hence,  $a \lds{\Low(\m P)} b \in L$ and  $b \rds{\Low(\m P)} a \in L$.
 
\item[(ii)${}\Ra{}$(iii):] Suppose $L$ satisfies (ii).  Let  $a\in \Low(\m P)$ and $b\in L$.  Since $\Low(\m P)$ is a join-extension of $\m P$, there exists $X\subseteq P$ such that $a = \Join[\Low(\m P)]X$. Hence,
\[
a \lds{\Low(\m P)} b 
  = \Bigl(\Join[\Low(\m P)] X \Bigr) \Biglds{\Low(\m P)} b
  = \Meet[\Low(\m P)]_{p\in X} (p\lds{\Low(\m P)} b)
  = \Meet[\m{L}]_{p\in X} (p\lds{\Low(\m P)} b).
\] 

Indeed, the second equality above follows from Lemma~\ref{lem:residuation:joins:meets}, while the third equality follows from~(ii) and Lemma~\ref{lem:meet:faithful}. Thus, $a \lds{\Low(\m P)} b \in L$, and likewise, $b \rds{\Low(\m P)} a \in L$.
 
\item[(iii)${}\Ra{}$(i):] Suppose that $L$ is a nucleus-system of $\Low(\m P)$.  In view of Lemma~\ref{lem:nuclear:system}, $\m{L}$ is a  residuated lattice.   We need to prove that the multiplication  of $\m{L}$ extends the multiplication of $\m P$. We have observed before that $\m P$ is a  submonoid of $\Low(\m P)$. Since $P\subseteq L$, it follows that for every $x,y\in P$, $x\circ_{\m{L}} y=\gamma_L(xy) = xy$. Thus, $\m P$ is also a submonoid of $\m{L}$, as was to be shown.

We next prove that the said multiplication of $\m L$ is uniquely determined.  Indeed, suppose $\m{L}$ is given the structure of a residuated lattice with respect to multiplications $\ast$ and $\star$ that extend the multiplication of $\m P$ (to be indicated in the proof below as a juxtaposition). As $\ast$ and $\star$ are residuated, they preserve arbitrary joins in $\m L$.  Consider elements $x, y\in L$. Then there exist nonempty subsets $X, Y$ of $P$ such that $x=\Join[\m L] X$ and $y=\Join[\m L] Y$. Then $x\ast y=\Join[\m L] X \ast \Join[\m L] Y=\Join[\m L] \{x\ast y \mid x\in X, y\in Y\}=\Join[\m L] \{xy \mid x\in X, y\in Y\}=\Join[\m L] \{x\star y \mid x\in X, y\in Y\}=\Join[\m L] X \star \Join[\m L] Y=x\star y.$ Lastly, that the inclusion map $\m{P}\hookrightarrow \m{L}$ preserves all existing residuals and meets follows from Lemma~\ref{residuals} and Lemma~\ref{lem:meet:faithful}, respectively. \qedhere
\end{itemproof}
\end{proof}

We note that Theorem~\ref{thm:completions-mainresult} provides a simple proof of the fact that the Dedekind-MacNeille completion of a residuated partially ordered monoid is a residuated lattice.

\begin{corollary}\label{Dedekind-MacNeille}
The Dedekind-MacNeille completion $\McN(\m{P})$ of a residuated partially ordered monoid $\m P$ can be uniquely endowed with the structure of a residuated lattice with respect to a unique multiplication that extends the multiplication of $\m P$.
\end{corollary}   

\begin{proof}
Let $a\in P$ and  $b\in \McN(\m{P})$. In view of Theorem~\ref{thm:completions-mainresult}(ii), it will suffice to show that $a \lds{\Low(\m P)} b \in \McN(\m{P})$ and $b \rds{\Low(\m P)} a \in \McN(\m{P})$. We just prove that $a \lds{\Low(\m P)} b \in \McN(\m{P})$. By Banaschewski's~\cite{Ban56} aforementioned result, $\McN(\m{P})$ is the only join- and meet-completion of $\m P$. Hence, there exist a subset $X$ of $P$ such that $b=\Meet[\McN(\m{P})] X$. Note that $b=\Meet[\McN(\m{P})] X=\Meet[\Low(\m P)] X$, by Lemma~\ref{lem:meet:faithful}. Invoking Lemmas~\ref{lem:residuation:joins:meets} and~\ref{residuals}, we have:
\begin{align*}
a \lds{\Low(\m P)} b &= a \Biglds{\Low(\m P)} \Meet[\Low(\m P)] X
  = \Meet[\Low(\m P)]_{x\in X} (a\lds{\Low(\m P)} x) 
  = \Meet[\McN(\m{P})]_{x\in X} (a\lds{\Low(\m P)} x) \\
  &= \Meet[\McN(\m{P})]_{x\in X} (a\lds{\m P} x)\in \McN(\m{P}).\qedhere
\end{align*} 
\end{proof}

As special case of Theorem~\ref{thm:completions-mainresult} occurs when $\m P$ is a meet-semilattice, that is, a partially ordered monoid whose multiplication is the meet operation. Then one can consider join-completions $\m L$ of $\m P$ that are Heyting algebras with respect to their lattice reducts. An example in point is the largest join-completion $\Low(\m P)$ of $\m P$. Denoting the Heyting implication of $\m L$ by $\tos{\m{L}}$, we have the following consequence of Theorem~\ref{thm:completions-mainresult}:

\begin{corollary}\label{completions-heyting}
Let $\m P$ be a meet-semilattice with top element $\ut$ and $\m{L}$ a join-completion of $\pair{P,\leq}$. Then, the following statements are equivalent:
\begin{enumerate}[(i)]
\item\label{cormainresult-item1}  $\m{L}$ is a Heyting algebra with respect to the lattice reduct of $\m{L}$ (more precisely, it can be made into a Heyting algebra by adding the Heyting implication). 
\item\label{cormainresult-item2} For all $a \in P$ and $b \in L$, $a \tos{\Low(\m P)} b \in L$. 
\item\label{cormainresult-item3} $L$ is a nucleus-system of  $ \pair{\Low(\m P),\mt,\jn,\mt,\tos{\Low(\m P)},\ut}$. 
\item\label{cormainresult-item4} $\gamma_L$ is a nucleus on $\Low(\m P)$. 
\end{enumerate}
Furthermore, whenever the preceding conditions are satisfied, the Heyting algebra structure on $\m L$ is uniquely determined and the inclusion map $\m{P}\hookrightarrow \m{L}$ preserves all existing meets and residuals.
\end{corollary}

Using Corollary~\ref{completions-heyting} and arguing as in the proof of Corollary~\ref{Dedekind-MacNeille}, we get an alternative proof of the  well-known fact that the Dedekind-MacNeille completion of an implicative semilattice is a Heyting algebra. We use the term \emph{implicative semilattice} for a meet-semilattice whose meet operation is residuated.

\begin{corollary}
The Dedekind-MacNeille completion of an implicative semilattice is a Heyting algebra.
\end{corollary}

In Section~\ref{sec:FEP:HRL}, the proof of finite embeddability property for  a number of varieties of residuated lattices requires the combination of the settings described in Theorem~\ref{thm:completions-mainresult} and Corollary~\ref{completions-heyting}.  More specifically, consider an integral \emph{meet-semilattice ordered monoid} $\m P$. The pertinent question here is: Which join-completions $\m L$ of $\m P$ are both residuated lattices and also Heyting algebras with respect to their lattice reduct. An example in point is the largest join-completion $\Low(\m P)$ of $\m P$. In such a situation, one can extend the language of $\m L$ with the addition of a Heyting implication $\tos{\m{L}}$. The next result, which is an immediate consequence of Theorem~\ref{thm:completions-mainresult}, provides a description of such completions:

\begin{corollary}\label{completions-RL-Heyting}
Let $\m P$ be an integral meet-semilattice monoid and $\m{L}$ a join-completion of $\pair{P,\leq}$. Then, the following statements are equivalent:
\begin{enumerate}[(i)]
\item\label{RL-Heyting-item1} {$\m{L}$ can be given both a structure of a residuated lattice whose multiplication extends the multiplication of $\m P$ and of a Heyting algebra with respect to the lattice reduct of $\m{L}$. }
\item\label{RL-Heyting-item2} For all $a \in P$ and $b \in L$, $a \lds{\Low(\m P)} b \in L$, $b \rds{\Low(\m P)} a \in L$ and $a \tos{\Low(\m P)} b \in L$. 
\item\label{RL-Heyting-item3} $L$ is a nucleus-system of the algebras $ \pair{\Low(\m P),\mt,\jn,\cdot,\lds{\Low(\m P)},\rds{\Low(\m P)}, \ut}$ and \break $ \pair{\Low(\m P),\mt,\jn,\mt,\tos{\Low(\m P)},\ut}$. 
\item\label{RL-Heyting-item4} The closure operator $\gamma_L$ is a nucleus on   $\pair{\Low(\m P),\mt,\jn,\cdot,\lds{\Low(\m P)},\rds{\Low(\m P)}, \ut}$ and on $ \pair{\Low(\m P),\mt,\jn,\mt,\tos{\Low(\m P)},\ut}$. 
\end{enumerate}
Furthermore, whenever the preceding conditions are satisfied, the two structures are uniquely determined and the inclusion map $\m{P}\hookrightarrow \m{L}$ preserves multiplication, all existing residuals (including Heyting implication) and meets.
\end{corollary}

\section{Join-completions of involutive ordered algebras}
\label{s:DM-completion}

This section focuses on a number of interesting applications in the setting of involutive ordered algebras.  Two results of particular interest are Theorem~\ref{thm:involuted2} and Theorem~\ref{thm:integrallyclosed}. The former is concerned with the nucleus-systems of a residuated partially ordered monoid $\m P$ that are involutive and provides a construction of their Dedekind-MacNeile completion within any join-completion of $P$. Theorem~\ref{thm:integrallyclosed} provides a  succinct and computation-free proof of the fact that the Dedekind-MacNeile completion of an Archimedean partially ordered group is a conditionally complete partially ordered group.

An element $d$ of a residuated partially ordered monoid $\m P$ is called \emph{cyclic} if for all $x \in P, d\rd x = x\ld d$. If $d$ is a cyclic element of $\m P$, we denote by $x \ras d$ both residuals of $d$ by $x$. We leave the proof of the following simple result to the reader.

\begin{lemma}\label{lem:specialnuclei}
If $d$ is a cyclic element of a residuated partially ordered monoid $\m P$, then the map $\gamma_d\colon x\mapsto (x \ras d)\ras d$ is a nucleus whose associated nucleus-system is $P_{\gamma_d} = \{x\ras d \mid x\in P\}$.
\end{lemma}

A \emph{cyclic dualizing} element of a residuated partially ordered monoid $\m P$ is a cyclic element $d$ that satisfies $\gamma_d(x)=x$, for all $x\in P$. An \emph{involutive residuated lattice}\footnote{In the literature the term \emph{cyclic involutive} is used for this notion, as \emph{noncyclic involutive} residuated lattices have also been studied.} is an algebra $\m{L} = \lan L, \mt, \jn, \cdot, \ld, \rd, \ut, d \ran$ such that $\m{L}' = \lan L, \mt, \jn, \cdot, \ld, \rd, \ut\ran$ is a residuated lattice and $d$ is a cyclic dualizing element of $\m L'$. If in the preceding definition we replace `lattice' by `partially ordered monoid,' we obtain the concept of an \emph{involutive residuated partially ordered monoid}. 

\begin{remark}\label{rem:char:involutive:RL}
 The choice of the term `involutive' reflects the fact that the map $x \mapsto x\ras d$ is an involution of the underlying order structure. In fact, it can be easily shown that an involutive residuated lattice is term equivalent to an algebra $\m{L} = \lan L, \mt, \jn, \cdot, \ut, \,' \ran$ such that 
\begin{enumerate} 
\item $\lan L, \cdot, \ut\ran$ is a monoid,
\item $\lan L, \mt, \jn,\,' \ran$ is an involutive lattice, and 
\item $xy \leq z \iff y \leq (z'x)' \iff x \leq (yz')'$, for all $x, y, z \in L$. 
\end{enumerate}
 \end{remark}
 
The following result is in the folklore of the subject and generalizes the classical Glivenko-Frink Theorem for Brouwerian lattices (\cite{Gli29} and~\cite{Fri62}; see also~\cite{Sch77}*{Theorem~4.1}, and~\cite{Ros90b}*{p.~142}). We make use of the notation of Lemma~\ref{lem:specialnuclei}.
 
\begin{lemma}\label{lem:involuted1}
Let $\m P$ be a residuated partially ordered monoid and $\gamma$ a nucleus on $\m P$. Then $\m P_\gamma$ is an involutive residuated partially ordered monoid if and only if there exists a cyclic element $d$ of $\m P$ such that $\gamma = \gamma_d$. \end{lemma}
 
\begin{proof} Suppose first that $d$ is a cyclic element of $\m P$. In light of Lemma~\ref{lem:specialnuclei}, $\gamma_d$ is a nucleus on $\m P$ and $d =\ut\ras d \in P_{\gamma_d}$. Lastly, it is clear that $d$ is a cyclic dualizing element of $\m P_{\gamma_d}$.
 
Conversely, assume that $\gamma$ is a nucleus on $\m P$ such that the nucleus-system $\m{P}_\gamma$ is an involutive residuated partially ordered monoid. Let $d$ be the cyclic dualizing element of $\m P_\gamma$. We claim that it is a cyclic element of $P$. Indeed, let $x\in P$. Invoking Lemma~\ref{lem:nucleus}, we have $x\ld d = \gamma(x)\ld d = d\rd \gamma(x) = d\rd x$. Making use of Lemma~\ref{lem:nucleus} once more and writing $x\ras d$ for the common value $x\ld d =d\rd x$ for all $x\in P$, we get  
\[
\gamma_d(x) = (x\ras d)\ras d = (\gamma(x)\ras d)\ras d = \gamma(x).
\] 
Thus, $\gamma = \gamma_d$. 
\end{proof}

The next result generalizes Theorem~4.3 of~\cite{Sch77}.  It is a  far reaching generalization of the  Glivenko-Stone Theorem (\cite{Gli29},~\cite{Sto36}), which asserts that the  Dedekind-MacNeille  completion of a Boolean algebra is a Boolean algebra.
 
\begin{theorem}\label{thm:involuted2}
Let $\m P$ be a residuated partially ordered monoid and let $\m{L}$ be  a join-completion of $\m P$ which is a residuated lattice with  respect to a multiplication that extends the multiplication  of $\m P$ (see Theorem~\ref{thm:completions-mainresult}). Then for every cyclic dualizing element $d \in P$,  $\m{L}_{\gamma_d}$ is the Dedekind-MacNeille completion of $\m{P}_{\gamma_d}$.
\end{theorem}

\begin{proof}
Let $d$ be a cyclic element of $\m P$. Note first that $d$ is a cyclic element of $\m L$.  Indeed, let $a \in L$. There exists $X\subseteq P$ such that $a =  \Join[\m{L}]X$. Hence, by Lemma~\ref{residuals},  
\begin{align*}
a\lds{\m{L}} d &= \Big(\Join X\Big)\Biglds{\m L} d = \Meet_{x\in X}(x\lds{\m{L}} d) =  \Meet_{x\in X}(x\lds{\m{P}} d) \\
& = \Meet_{x\in X}(d \rds{\m{P}} x)  = \Meet_{x\in X}(d \rds{\m{L}} x) = d \rds{\m{L}} a.
\end{align*}

We complete the proof by showing that $\m{L}_{\gamma_d}$ is a  join- and meet-completion of $\m{P}_{\gamma_d}$. As  the map $x \mapsto x\ras d$ is an involution of  $\m{L}_{\gamma_d}$  by Lemma~\ref{lem:involuted1}, it will suffice to show that $\m{L}_{\gamma_d}$  is a meet-completion of $\m{P}_{\gamma_d}$. To this end,  let $a \in L_{\gamma_d}$.  By Lemma~\ref{lem:specialnuclei}, there exists  an element $b \in L$ such that $a = b\ras_{\m L} d$. Also,  there exists $X\subseteq P$ such that $b = \Join[\m{L}]X$. Thus,  $a = \Meet_{x\in X}(x\ras_{\m{L}} d) =  \Meet_{x\in X}(x\ras_{\m{P}} d)$.  We have shown that  $\m{L}_{\gamma_d}$ is a meet-completion of $\m{P}_{\gamma_d}$.
\end{proof}

Theorem~\ref{thm:involuted2} subsumes  and illuminates the following construction of the De\-de\-kind-MacNeille completion of an  involutive residuated partially ordered monoid proposed in~\cite{Ros90a} (see also ~\cite{Ros90b}*{p.~147}). 
 
\begin{corollary}\label{construction}
Let $\m P$ be an involutive residuated partially ordered monoid with cyclic  dualizing element $d$. Then  $\Low(\m P)_{\gamma_d}$ is the Dedekind-MacNeille completion  of $\m{P}_{\gamma_d}$.
\end{corollary}
 
We also have the following result as an immediate consequence of Lem\-ma~\ref{lem:involuted1} and Theorem~\ref{thm:involuted2}.
 
\begin{corollary}
Let $\m P$ be an involutive residuated partially ordered monoid.  
\begin{enumerate}[(i)]
\item The Dedekind-MacNeille completion $\McN(\m{P})$ of $\m P$ is an involutive residuated lattice.  
\item The inclusion map $\m{P} \hookrightarrow \McN(\m{P})$ preserves  products, residuals, and all existing meets and joins.  
\end{enumerate}  
\end{corollary}
 
It is important to mention that among all join-completions of an involutive residuated partially ordered monoid $\m P$, the Dedekind-MacNeille completion is the only one that is an involutive residuated lattice with respect to a multiplication that extends the multiplication of $\m P$.  This follows from the fact that such a join-completion is also a meet-completion, and as we noted in the previous section, this characterizes the Dedekind-MacNeille completion.

We close this section with a discussion of the Dedekind-MacNeille completion of a partially ordered group. A \emph{partially ordered group} is a partially ordered monoid $\m P$ in which every element $x$ has a two-sided inverse $x^{-1}$. Such a structure is an involutive residuated monoid. The division operations are given by $x\ld y=x^{-1}y$ and $y\rd x=yx^{-1},$ for all $x, y\in P$, while the unit $\ut$ is the unique cyclic dualizing element of $\m P$.
By the preceding results, the Dedekind-MacNeille completion of a partially ordered group is an involutive residuated lattice.  However, this completion is of little interest from the point of view of ordered groups, since the presence of least and  greatest elements prevents ${\McN(\m{P})}$ from being even a partially ordered group. Let $\McN*(\m P)$ denote the involutive partially ordered monoid obtained from $\McN(\m{P})$ by removing its least and greatest elements. Note that $\McN*(\m P)$ is \emph{conditionally complete}. This means that the join (meet) of any upper (lower) bounded subset of $\McN*(\m P)$ exists. Note further that it is a lattice precisely when $\m P$ is directed, that is, when any two elements of $\m P$ have an upper and a lower bound. 

The natural question arises as to when $\McN*(\m P)$ is a partially ordered group. By L.V.~Kantorovitch's  well-known result (see~\cite{Fuc63}*{p.~90}), every conditionally complete partially ordered group is Archimedean. Recall that a partially ordered group $\m P$ is \emph{Archimedean} if for every $x, y \in P$, the inequalities $x^n \leq y$ ($n = 1, 2, \dots$) imply $x\leq \ut$. Thus, if $\McN*(\m P)$ is a partially ordered group, then $\m P$ is Archimedean. The following important result\,---\,due to Krull, Lorenzen, Clifford, Everett, and Ulam (see~\cite{Fuc63}*{p.~95}, for original references)\,---\,shows that the converse is also true. We can use the general theory developed earlier to provide a substantially shorter and conceptually simpler proof than the existing ones in the literature (see, for example, \cite{Fuc63}*{pp.~92-95} or ~\cite{Gla99}*{pp.~191-194}, where the proof in~\cite{Fuc63} is reproduced).
 
\begin{theorem}\label{thm:integrallyclosed}
If $\m P$ is an Archimedean partially ordered group, then $\McN*(\m P)$ is a conditionally complete partially ordered group. If in addition $\m P$ is directed, then $\McN*(\m P)$ is a conditionally complete lattice-ordered group. 
\end{theorem}
 
\begin{proof}
In view of the preceding discussion, it will suffice to prove  that every element of $\McN*(\m P)$ is  invertible whenever $\m P$ is Archimedean. For each  $x \in \McN*(\m P)$, let $x' = x\ld \ut$.  We have already  observed that $\McN*(\m P)$ is an involutive  residuated partially ordered monoid, that is, the map $x \mapsto x'$ is an  involution.

Assume that $\m P$ is Archimedean  and let $a$ be an arbitrary element of $\McN*(\m P)$.  We need to show that $aa' = \ut$, which is equivalent to $\ut\leq aa'$. As $\m P$ is meet-dense in  $\McN*(\m P)$, it will suffice to prove  that every element of $\ua(aa')\cap P$ exceeds $\ut$. To  this end, let $x \in \ua(aa')\cap P$. The inequality $aa'\leq x$ implies $a' \leq (x'a)'$, by Remark~\ref{rem:char:involutive:RL}, and therefore $x'a\leq a$. It follows that $(x')^2 a\leq x' a\leq a$, and inductively $(x')^n a\leq a$, for all $n\in \Z^+$.  
Next, due to the join-density and meet-density of $\m P$ in $\McN(\m P)$, there exist $u,w\in P$ such that $w\leq a\leq u$. The inequalities $(x')^n a \leq a$ ($n\in \Z^+$) immediately yield the inequalities $(x')^n \leq u\rd w$. As  $\m P$ is Archimedean, we obtain $x'\leq \ut$, and thus $\ut \leq x$. This completes the proof of $\ut\leq a'a$ and so $a'a=\ut$.

Lastly, it is clear that if $\m P$ is directed, then $\McN*(\m P)$ is a conditionally complete lattice-ordered group. 
\end{proof}

\begin{remark}
If $d$ is a cyclic dualizing element of a residuated lattice $\m L$, then $d = \ut \ras d = \ut'$. Thus, in every involutive residuated lattice $\m L=\pair{L, \mt, \jn, \cdot, \ld, \rd, \ut, d}$ we have that
\[
d=d' \quad\iff\quad \ut=\ut' \quad\iff\quad d=\ut.
\]
This can be used to characterize partially ordered groups as the involutive residuated lattices satisfying the equations $d=\ut$ and $x\ld x \eq \ut$. Indeed, every partially ordered group satisfies these equations. For the opposite direction, if $\m L$ satisfies these two equations, then for all $a\in L$, $\ut = a'\ld a' = (aa')'$ and therefore $\ut = \ut' = aa' = a(a\ld \ut)$, and analogously $(\ut\rd a)a = \ut$.

Nonetheless, there are examples of involutive residuated lattices with $d=\ut$ that do not satisfy the equation $x\ld x\eq \ut$. This is, for example, the case for every residuated lattice of the form $\McN*(\m P)$, where $\m P$ is a non-Archimedean partially ordered group. This also shows that the equation $x\ld x\eq \ut$ is not preserved by $\McN*$.
\end{remark}

\section{The finite embeddability property for residuated structures}
\label{sec:FEP:HRL}

A class $\cls K$ of algebras has the finite embeddability property (FEP, for short) if every finite partial subalgebra of any member of $\cls K$ can be embedded into a finite algebra of $\cls K$. This property has received considerable attention in the literature due to the fact that a number of decidability results about classes of algebras are consequences of it. (Refer to Section~\ref{s:fep} for a general discussion of these matters.) The most consequential study of the FEP for classes of residuated lattices is presented in the articles~\cite{BvA02} and~\cite{BvA05}, where it is shown, among other results, that the varieties of commutative integral residuated lattices and integral residuated lattices satisfy the FEP. On the other hand, the FEP is rare among non-integral varieties. For example, it is shown in~\cite{BvA02} that any variety of residuated lattices that contains the $\ell$-group $\Z$ of integers fails the FEP.  Other relevant articles include~\cite{BaNe72, Bus11, GJ13, Far08, MT46, OT99, vA05, vA09} and~\cite{vA11}.

In this section, we apply the results of Section~\ref{section<<joinextensions} to produce refined algebraic proofs of some of the existing results on the FEP. Further, we outline a method for establishing this property for algebras that involve more than one residuated operation. In particular, we prove that the variety $\cls{HRL}$ of Heyting residuated lattices, namely algebras that combine compatible structures of residuated lattices and Heyting algebras, satisfies the FEP. A direct consequence of the latter result is that the variety of distributive integral residuated lattices satisfies the FEP, a result that has been obtained independently in~\cite{Bus11} and~\cite{GaJi} by alternative means.

{In order to be as precise as possible, we devote Subsection~\ref{subsect:preliminaries:FEP} to preliminaries about partial algebras and the FEP. In Subsection~\ref{subsect:residuated:structures}, we set up a language that allows us to handle  at once structures with many residuated operations and their residuals. We also employ ideas due to Blok and van Alten (see~\cite{BvA05}) to prove that the term algebra with two binary operations can be given a divisibility order with respect to which the operations are residuated.  In Subsection~\ref{subsect:HRL:FEP}, the theory of Section~\ref{section<<joinextensions} is put to work to produce a "potentially" finite extension of a finite partial algebra. The method is illustrated for the proof of the FEP for the variety $\cls {HRL}$ of Heyting residuated lattices (see Lemma~\ref{lem:embedding}), but it easily applies to the results in ~\cite{BvA05}. The introduction of the Heyting arrow guarantees that our construction preserves lattice distributivity and, in particular, implies the FEP for the variety of distributive integral  residuated lattices. }

\subsection{Partial algebras, homomorphism, and the finite embeddability property}
\label{subsect:preliminaries:FEP}

An \emph{algebraic language} is a pair $\L = \lan L, \tau \ran$ consisting of a nonempty set $L$ of \emph{operation symbols} and a map $\tau \colon L \to \N$. The image of an operation symbol under $\tau$ is called its \emph{arity}.\footnote{According to our definitions, all the operation symbols are of finite arity. Thus, we will not consider \emph{nonfinitary} languages.} Nullary operation symbols are called \emph{constants}.

Let $P$ be any nonempty set, and let $k\in \N$. A \emph{$k$-ary partial operation} $f$ on $P$ is a map $f \colon R \to P$ from a subset $R$ of $P^k$ to $P$. We refer to $R$ as the \emph{domain} of $f$ and denote it by $\dom f$. If $\pair{p_1,\dots,p_k}\in \dom f$, then we say that $f$ is \emph{defined} at $\pair{p_1,\dots,p_k}$ or that $f(p_1,\dots,p_k)$ exists.\footnote{When we write an equality involving partial operations, we always intent to convey that both sides are defined and are equal.} Observe that for $k = 0$, $P^0 = \{\emptyset\}$ and therefore a partial nullary operation (a partial constant) $f$ on $P$ is either empty or distinguishes exactly one element of $P$. If $\dom f = P^k$, then $f$ is a \emph{total operation} on $P$.

Given a language $\L = \lan L,\tau\ran$, by a \emph{partial $\L$-algebra} $\m P$ we understand an ordered pair $\lan P, \{f^{\m P}\}_{f\in L}\ran$, where $P$ is a nonempty set and $f^{\m P}$ is a $\tau{(f)}$-ary partial operation on $P$ for each $f\in L$. The map $f^{\m P}$ is called the \emph{fundamental operation} of $\m P$ corresponding to $f$. A partial $\L$-algebra $\m P$ is called a \emph{total $\L$-algebra}, or simply, an \emph{$\L$-algebra}, if all its operations are total.

In what follows, we will drop the superscript $\m P$ of a partial operation $f^{\m P}$ whenever there is no danger of confusion. Likewise, we often drop the prefix $\L$ from term $\L$-algebra.

Let $\m P$ and $\m Q$ be two partial $\L$-algebras. A map $\varphi \colon P \to Q$ is called a \emph{homomorphism} from $\m P$ to $\m Q$, in symbols $\varphi \colon \m P \to \m Q$, if for every operation symbol $f\in L$ and every sequence 
$\pair{p_1,\dots,p_{\tau(f)}}\in P^{\tau(f)}$ for which $f^{\m P}(p_1,\dots,p_{\tau(f)})$ exists,  $f^{\m Q}(\varphi(p_1),\dots,\varphi(p_{\tau(f)}))$ exists and
\[
\varphi (f^{\m P} (p_1,\dots,p_{\tau(f)})) = f^{\m Q} (\varphi(p_1),\dots,\varphi(p_{\tau(f)})).
\]
By an \emph{isomorphism} from $\m P$ to $\m Q$ we mean a bijection $\varphi \colon P \to Q$ such that $\varphi \colon \m P \to \m Q$ and $\varphi^{-1}\colon \m Q\to\m P$ are homomorphisms. Clearly, if $\varphi \colon \m P \to \m Q$ is an isomorphism, then so is $\varphi^{-1}\colon \m Q\to\m P$. In this case, we say that $\m P$ and $\m Q$ are \emph{isomorphic} and write ${\m P} \cong {\m Q}$.

An \emph{embedding} of a partial algebra $\m P$ into another partial algebra $\m Q$ is an injective homomorphism. We say that $\m P$ is a \emph{partial subalgebra} of $\m Q$ if $P\subseteq Q$ and the inclusion map $i \colon {\m P} \to {\m Q}$ is a homomorphism, that is, an embedding. A partial subalgebra $\m P$ of $\m Q$ is called \emph{full} if for every $f\in L$ and $p_1,\dots,p_{\tau(f)}\in P$, if $f^{\m Q} (p_1,\dots,p_{\tau(f)})$ is defined and is an element of $P$, then $f^{\m P}(p_1,\dots,p_{\tau(f)}) = f^{\m Q}(p_1,\dots,p_{\tau(f)})$. Therefore, a full partial subalgebra of a partial algebra $\m Q$ is determined by its underlying set $P$, and will be denoted by ${\m Q}\rest P$.

If $\varphi \colon {\m P} \to {\m Q}$ is a homomorphism, we define another partial subalgebra $\varphi[\m P]$ of ${\m Q}$ as follows: the underlying set is $\varphi[P]$; and if $f\in L$ and $p_1,\dots, p_{\tau(f)}\in P$ are such that $f^{\m P}(p_1,\dots,p_{\tau(f)})$ is defined, then we define
\[
f^{\varphi[\m P]}(\varphi(p_1),\dots,\varphi(p_{\tau(f)})) = \varphi(f^{\m P}(p_1,\dots,p_{\tau(f)})).
\]
It is clear that $\varphi[\m P]$ is a partial subalgebra of $\m Q$; we call it the \emph{image} of $\m P$ by $\varphi$. Furthermore, $\varphi \colon {\m P} \to \varphi[\m P]$ is a surjective homomorphism, and it is an isomorphism whenever $\varphi$ is an embedding. An embedding $\varphi\colon \m P\to \m Q$ is called \emph{full} if its image $\varphi[\m P]$ is a full partial subalgebra of $\m Q$. This is the case if and only if $\m P\cong \varphi[\m P] = \m Q\rest \varphi[P]$.

From now on, we suppose that $\cls K$ is a class of $\mathcal{L}$-algebras.

\begin{definition}\label{def:FEP}
An algebra $\m A$ is said to have the \emph{finite embeddability property} in $\cls K$ (FEP, for short) if every finite partial subalgebra of $\m A$ can be embedded into a finite algebra of $\cls K$. The class $\cls K$ is said to have the FEP if every algebra in $\cls K$ has the FEP in $\cls K$.
\end{definition}

The next definition introduces a related version of the FEP, which is often mentioned in the literature under the same name.

\begin{definition}\label{def:full:FEP}
An algebra $\m A$ is said to have the \emph{full finite embeddability property} in $\cls K$ (FEP\textss{+}, for short) if every finite full partial subalgebra of $\m A$ can be fully embedded into a finite algebra of $\cls K$. The class $\cls K$ is said to have the FEP\textss{+} if every algebra in $\cls K$ has the FEP\textss+ in $\cls K$.
\end{definition}

An alternative formulation of the FEP, which states that ``every finite full partial subalgebra of $\m A$ is embeddable in some finite member of $\cls K$" is clearly equivalent to the definition of FEP above. It is also clear that the FEP\textss{+} implies FEP for any class of algebras. While the other implication does not hold in general, see Example~\ref{ex:FEP:notFEPp}, the next lemma shows that the two properties are equivalent when the language is finite.

\begin{lemma}\label{lem:FEP:FEPp}
Let $\cls K$ be a class of $\L$-algebras and $\m A$ an $\L$-algebra.
\begin{enumerate}[(i)]
 \item If $\m A$ has the FEP\textss{+} in $\cls K$, then it has the FEP in $\cls K$.
 \item If $\L$ is finite and $\m A$ has the FEP in $\cls K$, then it has the FEP\textss+ in $\cls K$.
\end{enumerate}
\end{lemma}

\begin{proof}
For the first part, notice that if $\m P$ is a finite partial subalgebra of an algebra $\m A$, then the identity map $i \colon P \to P$ is an embedding of $\m P$ into ${\m A}\rest P$. By hypothesis, there is a full embedding $\varphi \colon {\m A}\rest P \to {\m C}$, where ${\m C}$ is a finite algebra in $\cls K$. Hence, $\varphi\, i$ is an embedding of $\m P$ into ${\m C}$.

With regard to the second part, suppose that $\m P$ is a finite full partial subalgebra of $\m A$ and consider the set $Q = P\cup\{f^{\m A}(p_1,\dots, p_{\tau(f)}) \mid p_1,\dots,p_{\tau(f)} \in P,\ f\in L\}$. As $L$ and $P$ are finite, so is $Q$. Let $\m Q = \m A\rest Q$ be the full partial subalgebra of $\m A$ generated by $Q$. Because $\m A$ has the FEP in $\cls K$, there exists a finite algebra $\m C \in \cls K$ and an embedding $\varphi \colon \m Q \to \m C$. We claim that $\varphi\rest_P \colon \m P\to \m C$ is a full embedding of $\m P$ into $\m C$. Obviously, it is an embedding, and therefore we only need to show that $\varphi\rest_P [\m P]$ is a full partial subalgebra of $\m C$. To this end, let $f\in L$, $c_1,\dots,c_{\tau(f)}\in \varphi[P]$, and $f^{\m C}(c_1,\dots,c_{\tau(f)})\in \varphi[P]$. We need to prove that  $f^{\varphi\rest_P[\m P]}(c_1,\dots,c_{\tau(f)})$ is defined. There exist $p_1,\dots,p_{\tau(f)}, p \in P$ such that $c_i = \varphi(p_i)$ (for all $i$) and $\varphi(p)= f^{\m C}(c_1,\dots,c_{\tau(f)})$. This implies that $f^\m A(p_1,\dots,p_{\tau(f)})\in Q$. Hence, $f^\m Q(p_1,\dots, p_{\tau(f)}) = f^\m A(p_1,\dots, p_{\tau(f)})$, and as $\varphi$ is an embedding of $\m Q$ in $\m C$, we have that $\varphi(f^\m Q (p_1,\dots, p_{\tau(f)})) = f^\m C (\varphi (p_1),\dots, \varphi(p_{\tau(f)})) = f^\m C (c_1,\dots,c_{\tau(f)}) = \varphi(p)$. But then $p = f^\m Q(p_1,\dots,p_{\tau(f)}) = f^\m A(p_1,\dots,p_{\tau(f)})$. As $\m P$ is a full partial subalgebra of $\m A$, $f^\m P(p_1,\dots,p_{\tau(f)})$ is defined and is equal to $p$. Hence $f^{\varphi\rest_P[\m P]}(c_1,\dots,c_{\tau(f)})$ is defined.
\end{proof}

The following example shows that the assumption of the finiteness of the signature is essential for the equivalence of FEP and FEP\textss+.

\begin{example}\label{ex:FEP:notFEPp}
Start with the language $\L$ with $L=\{f_n \mid n\in \N\}$, consisting of unary operation symbols. Let $\m A = \pair{\N,\{f_n^\m A\}_{n\in\N}}$, with each $f_n$ defined by $f_n^{\m A}(k) = k+n$, for all $k\in \N$. Lastly, consider the class $\cls K = \{\m C_m \mid m\in \N\}$, consisting of the finite algebras $C_m =\Z/m\Z$ with $f_n^{\m C_m} ([k]) = [k]+[n]$, for all $m,k,n\in \N$.
It is easy to see that every finite partial subalgebra $\m P$ of $\m A$ can be embedded in an algebra of $\cls K$: in particular, if $k$ is the maximum of $P$, then $\m P$ is embeddable in $\m C_{k+1}$. On the other hand, $\m P$ cannot be fully embedded in any $\m C_m$, since given any embedding $\varphi \colon \m P\to \m C_m$, we can consider $sm > k$, where $k$ is the maximum of $P$, and therefore $f_{sm}^{\m P}$ is not defined for any element of $P$, but $f_{sm}^{\m C_m}$ is be the identity in $C_m$.
\end{example}

\subsection{Residuated structures}
\label{subsect:residuated:structures}

A \emph{residuated structure} is a structure $\m A$ that comprises a partial order $\leq$, and a set of residuated operations on $\pair{A,\leq}$, called \emph{multiplications}, and their residuals among their fundamental operations. The structure might carry other fundamental operations and constants as well. It will be clear from the context which operation symbols we use for the multiplications. Following our previous practice, we use symbols like $\ld$ and $\rd$ to represent the left and right residuals of a multiplication, and symbols like $\ras$ (or $\to$) for the residual of a commutative multiplication. 

A \emph{residual term of depth} 1 is a term of the form $x\ld z$, $z\rd x$, or $x\ras z$, where $x$ and $z$ are different variables; we call $z$ the \emph{central variable} of the term. A \emph{residual term of depth} $n+1$ is a term of the form $x\ld t$, $t\rd x$, or $x\ras t$, where $t$ is a residual term of depth $n$, and $x$ is a variable not appearing in $t$; its \emph{central variable} is the central variable of $t$. A \emph{multiplicative term of depth 1 with central variable $y$} is a term of the form $x\cdot y$, $y\cdot x$, or $x\mt y$. A \emph{multiplicative term of depth $n+1$ with central variable $y$} is a term of the form $x\cdot t$, $t\cdot x$, or $x\mt t$, where $t$ is a multiplicative term of depth $n$ with central variable $y$, and $x$ is a variable not appearing in $t$. We denote by $R_\L$ and $M_\L$ the set of all residual terms and multiplicative terms on a particular language $\L$, respectively.

For distinct variables $x_1,\dots,x_n, y, z$, $\rho(x_1,\dots,x_n,z)$ will denote any residual term of depth $n$ in these variables with central variable $z$, and $\lambda(x_1,\dots,$ $x_n,y)$ will denote any multiplicative term in these variables with central variable $y$. It should be noted that these notations are somewhat ambiguous since the term they represent depends on the choice and order of appearance of the operation symbols. However, this ambiguity will not create any confusion in the ensuing discussion.

It is easy to see that, given a residuated structure $\m A$, any multiplicative term $\lambda(x_1,\dots,x_n,y)$ defines a residuated map on $\pair{A,\leq}$ in each coordinate, and in particular in the central variable. That is to say, for every $a_1,\dots,a_n\in A$, $\lambda^\m A(a_1,\dots,a_n,y)\colon A\to A$ is residuated, and therefore it has a residual. Conversely, every residual term  $\rho(x_1,\dots,x_n,z)$ defines a residual map in its central variable, i.e., for every $a_1,\dots,a_n\in A$, $\rho^\m A(a_1,\dots,a_n,z)\colon A\to A$ is the residual of some residuated map. We make this precise in the next lemma.

\begin{lemma}\label{lem:backandforth}
Let $x_1,\dots, x_n, y, z$ be distinct variables. Given a residual term $\rho(x_1,\dots, x_n, z)$, there exists a multiplicative term $\lambda(x_1,\dots, x_n, y)$, and given a multiplicative term $\lambda(x_1,\dots, x_n, y)$ there is a residual term $\rho(x_1,\dots, x_n, z)$, such that for every residuated structure $\m A$ and $a_1,\dots,a_n,b,c\in A$,
\[
\lambda^\m A (a_1,\dots,a_n, b)\leq c \quad\iff\quad b \leq \rho^\m A (a_1,\dots, a_n, c).
\]
\end{lemma}

\begin{proof}
We proceed by induction on the depth of $\rho(x_1,\dots, x_n, z)$ and the depth of $\lambda(x_1,\dots,x_n,y)$, respectively.
\end{proof}

The following result follows immediately from the general theory of residuated maps, and therefore we omit its proof.

\begin{corollary}\label{cor:preservation:joins:meets}
Let $\m A$ be a residuated structure, $x_1,\dots, x_n, y, z$ distinct variables, $\rho(x_1,\dots, x_n, z)$ a residual term and $\lambda(x_1,\dots,x_n,y)$ a multiplicative term. Then,
\begin{enumerate}[(i)]
\item $\lambda^\m A$ respects arbitrary existing joins: for every $\{a_1,\dots,a_n\}\cup B\subseteq A$, if $\Join[\m A] B$ exists then
\[
\lambda^\m A (a_1,\dots,a_n,\Join[\m A] B) = \Join[\m A]_{b\in B} \lambda^\m A (a_1,\dots,a_n,b).
\]
\item $\rho^\m A$ respects arbitrary existing meets in its central variable: for every set of elements $\{a_1,\dots,a_n\}\cup B\subseteq A$, if $\Meet[\m A] B$ exists then
\[
\rho^\m A \Big(a_1,\dots,a_n,\Meet[\m A] B\Big) = \Meet[\m A]_{b\in B} \rho^\m A (a_1,\dots,a_n,b).
\]
\end{enumerate}
\end{corollary}

Consider now the term algebra $\m{T} = \pair{ T, \cdot, \mt}$ on the language $\{\cdot, \mt\}$ over a nonempty set of variables $X$. Let $F = T\cup\{\ut\}$ with $\ut\notin T$ and let extend to $F^2$ the operations in $\m T$ as follows: $\ut\cdot t = t\cdot \ut = t$, and $\ut\mt t = t\mt \ut = t$, for all $t\in F$. Further, we define the following relation on $F$: $s\leqc t$ if and only if whenever some occurrences of variables in $s$ are replaced by $\ut$, $s$ reduces to $t$ by application of the preceding equalities, and also $\ut \leqc \ut$.

It is immediate to see that $\leqc$ is a partial order on $F$ with top element $\ut$. Furthermore, both operations $\cdot$ and $\land$ preserve $\leqc$ and, in fact, they are residuated. To see this, we first observe the following:

\begin{lemma} \label{l:rieszproperty}
Retaining the notation of the preceding paragraph, let $\ast\in\{\cdot, \mt\}$, and let $r, s, t\in T$ such that $r\ast s\leqc t$. If $r\nleqc t$ and $s\nleqc t$, then there exist unique $r', s'\in T$ such that $r\leqc r'$, $s\leqc s'$, and $t = r'\ast s'$.
\end{lemma} 

\begin{proof}
Since $r\ast s\leqc t$, the elements $r$ and $s$ reduce to some $r'$ and $s'$ such that $t = r'*s'$. Thus, we have $t = r'\ast s'$, with $r\leqc r'$, $s\leqc s'$, and $t, r', s'\in T$. Since $r\nleqc t$ and $s\nleqc t$, the reductions $r'$ and $s'$ must be different from $\ut$, so they belong to $T$. The elements $r'$ and $s'$ are unique because if $r''$ and $s''$ were two elements with the same properties of $r'$ and $s'$, respectively, then $r'*s' = r''*s''$, and equality holds in the term algebra if and only if the two terms are syntactically equal.
\end{proof}

\begin{proposition}\label{prop:F}
The operations $\cdot$ and $\mt$ are residuated on $\pair{F,\leqc}$. Thus, $\m F = \pair{ F, \leqc, \cdot, \mt, \ld, \rd, \lds\mt, \rds\mt, \ut}$ is a residuated structure.
\end{proposition}

\begin{proof}
Let $r, t\in F$ and let $\ast\in\{\cdot, \mt\}$. We wish to describe the residual $r\ld_* t$. Observe that there exists some $s\in F$ such that $r\ast s\leqc t$ (for instance, $r\ast t\leqc t$), and hence we need to determine the largest such $s$. If $r\leqc t$ then $r*\ut \leqc t$, and given that $\ut$ is the maximum of the order $\leqc$, we have that $r\ld_* t = \ut$. If $r\nleqc t$ and there is $s\nleqc t$ such that $ r\ast s\leqc t$, then by Lemma~\ref{l:rieszproperty} there exist unique $r', s'\in T$ such that $r\leqc r'$, $s\leqc s'$, and $t = r'\ast s'$. In this case, it is easy to see that $r\ld_* t = s'$. Lastly, if $r\nleqc t$ and $s\leqc t$ whenever $ r\ast s\leqc t$, then $r\ld_* t = t$.
\end{proof}

Given an arbitrary algebraic language $\L$, a partial order $\leq$ on an $\L$-algebra $\m A$ is said to be a \emph{divisibility order} if for all non constant $f\in\L$ of arity $n\geq 1$,
\begin{enumerate}
\item $f^\m A(a_1,\ldots, a_n)\leq f^\m A(b_1,\ldots, b_n)$, whenever $a_i\leq b_i$ for all $i = 1,\ldots, n$, and
\item $f^\m A(a_1,\ldots, a_n)\leq a$, whenever $a_i\leq a$, for some $i = 1,\ldots, n$.
\end{enumerate}

It is immediate to see that the partial order $\leqc$ on $\m F$ is a divisibility order. By Higman's Lemma (see \cite{Coh81}*{page 123, Theorem 2.9}), any divisibility order is dually well-ordered, that is, it satisfies the Ascending Chain Condition\,---\,every strictly ascending sequence eventually terminates\,---\,and contains no infinite antichains, i.e., there is no infinite set of pairwise incomparable elements. Hence we have the following corollary:

\begin{corollary}\label{c:dualwellorder}
The partially ordered set $\pair{F, \leqc}$ is dually well-ordered.
\end{corollary}

\subsection{The finite embeddability property for residuated lattices with a Heyting implication}
\label{subsect:HRL:FEP}

A \emph{residuated lattice with a Heyting implication}, or \emph{Heyting residuated lattice} for short, is an algebra $\m A =\pair{A,\mt,\jn,\cdot,\ld,\rd,\to,\ut}$ that encompasses an integral residuated lattice and a Heyting algebra over the same underlying lattice. The main result of this section establishes that the class $\cls {HRL}$ of Heyting residuated lattices, which is clearly a variety, has the FEP.

Let $\m P$ be an integral partially ordered monoid and $D$ a nonempty subset of $P$.\footnote{Following our standard convention, we think of $P$ as a subset of $\Low(\m P)$.}  Consider $\ol D= \{\Meet X \mid X\subseteq D\}\subseteq \Low(\m P)$. Note that $\ol D$ contains the empty meet, which is $\ut$ because $\m P$ is integral, and furthermore $\ol D$ is closed under arbitrary nonempty meets in $\Low(\m P)$. It is therefore a closure system of $\Low(\m P)$. Moreover, any closure system of $\Low(\m P)$ containing $D$ must also contain $\ol D$. Thus, $\ol D$ is the smallest closure system of $\Low(\m P)$ containing $D$. We refer to $\ol D$ as the \emph{closure system} of $\Low(\m P)$ \emph{generated by} $D$.  The associated closure operator $\gamma_{\ol D}$ is given by $\gamma_{\ol D}(a)=\Meet(\ua a \cap D)$, for all $a\in \Low(\m P)$.

\begin{lemma}\label{l:keylemma}
Let $\m P$ be an integral partially ordered monoid and $D\subseteq P$ a nonempty set such that $a \lds{\Low(\m P)} b \in D \text{ and } b \rds{\Low(\m P)} a \in D$, for all $a \in  P$ and $b \in D$. The closure system $\ol D$ of  $\Low(\m P)$ generated by $D$ is a nucleus-system of $\Low(\m P)$. In particular, it is a residuated lattice with respect to the operations described in Lemma~\ref{lem:nuclear:system}.
\end{lemma}

\begin{proof} 
Recall that the inclusion map $\m P\hookrightarrow \Low(\m P)$ preserves the multiplication and all existing residuals and meets. Let $a\in\Low(\m P)$ and $b\in \ol D$. In view of Theorem~\ref{thm:completions-mainresult}, we only need to show that $a\lds{\Low(\m P)} b\in\ol D$ and $b\rds{\Low(\m P)} a\in \ol D$. Because our choice of $a$ and $b$, there exist $Y\subseteq P$ and $X\subseteq D$ such that $a=\Join Y$ and $b=\Meet X$. Hence, $a\lds{\Low(\m P)} b = \big(\Join Y\big)\biglds{\Low(\m P)} \big(\Meet X\big)=\Meet_{y\in Y}\Meet_{x\in X}(y\lds{\Low(\m P)} x)\in \ol D$, since $y\lds{\Low(\m P)} x\in D$, for all $y\in Y$ and $x\in X$. Likewise, $b\rds{\Low(\m P)} a\in \ol D$.
\end{proof}

We have the following consequence of Corollary~\ref{completions-RL-Heyting}:

\begin{corollary}\label{c:tokeylemma}
Let $\m P$ be an integral meet-semilattice monoid and $D\subseteq P$ a nonempty set such that $a \lds{\Low(\m P)} b$, $b \rds{\Low(\m P)} a$, and $a \tos{\Low(\m P)} b\in D$, for all $a \in P$ and $b \in D$. Then the closure system $\ol D$ generated by $D$ is a nucleus-system of $\pair{\Low(\m P),\mt,\jn,\cdot,\lds{\Low(\m P)},\rds{\Low(\m P)}, \ut}$ and  $\pair{\Low(\m P),\mt,\jn,\mt,\tos{\Low(\m P)},\ut}$. Equivalently, the associated closure operator $\gamma_{\ol D}$ is a nucleus on $\pair{\Low(\m P),\mt,\jn,\cdot,\lds{\Low(\m P)},\rds{\Low(\m P)}, \ut}$ and also of $ \pair{\Low(\m P),\mt,\jn,\mt,\tos{\Low(\m P)},\ut}$. 
\end{corollary}

The next two lemmas will lead us to Theorem~\ref{thm:HRL:FEP}, which establishes that the variety of Heyting residuated lattices enjoys the Finite Embeddability Property.

\begin{lemma}\label{lem:embedding}
Let $\m B$ be an arbitrary (not necessarily finite) partial subalgebra of a Heyting residuated lattice $\m A$. Then $\m B$ can be (order-) embedded into an order-complete algebra in $\cls{HRL}$.
\end{lemma}

\begin{proof}
Let $\m P$ be the $\{\mt, \cdot,\ut\}$-subreduct of $\m A$ generated by $B$, which is therefore an integral partially ordered monoid. Note that, whenever $u,v\in P$ and $u\lds{\m A} v \in P$, then $u\lds{\m A} v$ is the left residual of $v$ by $u$ in $\m P$, and analogously for $v\rds{\m A} u$ and $u\tos{\m A} v$. Thus, we also represent by $\m P$ the full partial subalgebra $\m A\rest P$ of $\m A$ determined by $P$. Thus,  $\m B\leq \m{P}\leq \m A$. Note that even if $B$ is finite, $P$ needs not be so. Consider the join-completion $\Low(\m P)$ of $\m P$ as an integral partially ordered monoid. We view $\m{P}$ as a subpartially ordered set of $\Low(\m P)$, and recall that the inclusion map preserves the multiplication, all existing residuals and Heyting arrows, and all existing meets.

Let $ D = \{\rho^{\Low(\m P)} (a_1,\dots, a_n, b) \mid \rho\in R_\L,\ a_1, \dots, a_n\in P,\ b\in B \}$. Note that $B\subseteq D$, since $\ut\in P$ and $\ut\ld b = b$, for all $b\in B$. Let $\ol D = \{\bigwedge X \mid X\subseteq D\}$ be the closure system of $\Low(\m P)$ generated by $D$. In light of Corollary~\ref{c:tokeylemma}, $\ol{D}$ is a nucleus-system of $\Low(\m P)$ relative to $\cdot\ds{\Low(\m P)}$ and $\mt\ds{\Low(\m P)}$. In particular, $\ol{\m D}$ is a Heyting residuated lattice. Furthermore, residuals, the Heyting implication, and arbitrary meets in $\ol{\m D}$ agree with those in $\Low(\m P)$.

We proceed to show that $\m B$ can be embedded into $\ol{\m D}$. So we prove the following for all $X\cup \{x, y\}\subseteq B$: if $x*\ds{\m A} y\in B$ then $x*\ds{\ol{\m D}} y = x*\ds{\m A} y$, for $*\in\{\cdot,\ld,\rd,\to\}$, if $\Join[\m A] X \in B$, then $\Join[\m A] X = \Join[\ol{\m D}] X$, and analogously for the meets.

Let $u, v\in B$ such that $u \cdot\ds{\m A} v\in B$. Then we have that $u \cdot\ds{\m P} v = u \cdot\ds{\m A} v \in B\subseteq \ol D$, and therefore $u \cdot\ds{\ol{\m D}} v = \gamma\ds{\ol D} (u \cdot\ds{\m P} v) = u \cdot\ds{\m P} v$. Thus, indeed $u \cdot\ds{\ol{\m D}} v = u \cdot\ds{\m A} v\in B$.

Next, let $u, v\in B$ such that $u \lds{\m A} v\in B$. Note that $u \lds{\m P} v = u \lds{\m A} v$ and $u \lds{\m P} v =  u \lds{\Low(\m P)} v = u \lds{\ol{\m D}} v$, since $\Low(\m P)$ is a join-completion of $\m{P}$ and $\ol D$ is a nucleus-system of $\Low(\m P)$. The same argument works for $\rd$ and $\to$.

Consider now $X\subseteq B$ such that $\Meet[\m A] X\in B$. Then $\Meet[\m P] X = \Meet[\m A] X$, since $\m P$ is a subpartially ordered set of $\m A$. But then $\Meet[\m P] X = \Meet[\Low(\m P)] X = \Meet[\ol{\m D}] X$, since $\Low(\m P)$ is a join completion of $\m{P}$, and therefore $\m P$ is meet-faithful in $\Low(\m P)$ (Lemma~\ref{lem:meet:faithful}), and $\ol D$ is a nucleus-system of $\Low(\m P)$.

Lastly, let $X\subseteq B$ such that $\Join[\m A] X\in B$. Note first that $\Join[\m B] X = \Join[\m P] X = \Join[\m A] X $. Thus, $\Join[\ol{\m D}] X = \gamma\ds{\ol D} \big(\Join[\Low(\m P)] X\big) \leq \gamma\ds{\ol D} \big(\Join[\m B] X\big) = \Join[\m B] X$. Therefore, $\Join[\ol{\m D}] X \leq \Join[\m A] X$. To prove the reverse inequality, let $\rho^{\Low(\m P)} (a_1,\dots, a_n, b)$\,---\,with $\rho\in R_\L$, $a_1,\ldots, a_n\in P$, and $b\in B$\,---\,be an upper bound of the elements of $X$ in $\ol{\m D}$, and let $\lambda\in M_\L$ be the corresponding multiplicative term given by Lemma~\ref{lem:backandforth}. Then, all $c\in B$,
\[
\lambda^{\Low(\m P)} (a_1,\dots, a_n, c)\leq b \quad\iff\quad c\leq \rho^{\Low(\m P)} (a_1,\dots, a_n, b).
\]
In particular, $\lambda^{\Low(\m P)} (a_1, \ldots, a_n, u)\leq b$, for all $u\in X$. Now, we know that $\lambda^{\Low(\m P)} \big(a_1, \ldots, a_n, \Join[\m A] X \big)\in P$, and therefore
\[
\lambda^{\Low(\m P)} \big(a_1, \ldots, a_n, \Join[\m A] X \big) = \lambda^{\m P} \big(a_1, \ldots, a_n, \Join[\m A] X \big) = \lambda^{\m A} \big(a_1, \ldots, a_n, \Join[\m A] X \big).
\]
By Corollary~\ref{cor:preservation:joins:meets}, we have
\begin{align*}
\lambda^{\Low(\m P)} \Big(a_1, \ldots, a_n, \Join[\m A] X\Big) &= \lambda^{\m A} \Big(a_1, \ldots, a_n, \Join[\m A] X\Big) \\
 &= \Join[\m A]_{u\in X} \lambda^{\m A} (a_1, \ldots, a_n, u)\\
 &= \Join[\m A]_{u\in X} \lambda^{\m P} (a_1, \ldots, a_n, u) \\
 &=  \Join[\m P]_{u\in X} \lambda^{\m P} (a_1, \ldots, a_n, u) \\
 &= \Join[\m P]_{u\in X} \lambda^{\Low(\m P)} (a_1, \ldots, a_n, u) \leq b.
\end{align*}
This implies that $\Join[\m A] X\leq \rho^{\Low(\m P)} (a_1,\dots, a_n, b)$, and so $\Join[\m A] X\leq \Join[\ol{\m D}] X$.
\end{proof}

The main ideas behind the proof of the next Lemma are due to Blok and van Alten (see~\cite{BvA05}).

\begin{lemma}\label{l:finiteness} 
With the notation of Lemma~\ref{lem:embedding} in effect, $\ol{\m D}$ is finite whenever $\m B$ is finite.
\end{lemma}
\begin{proof}

Let $B = \{b_1, \ldots, b_k\}$ be an enumeration of $B$ and consider $X = \{x_1, \ldots, x_k\}$, a set of $k$ distinct variables. Let also $\m F$ be as in Proposition~\ref{prop:F} and $\m{P}$ as in the proof of Lemma~\ref{lem:embedding}. 
Let $\varphi \colon F \to P$ be the $\{\mt, \cdot,\ut\}$-homomorphism that extends the assignment $x_i\mapsto b_i$. We think of $\varphi$ as a map $\varphi \colon F \to \Low(\m P)$, but keep in mind that $\varphi[F] = P$. It is important to observe that $\varphi$ is an order-homomorphism, because $\m P$ is integral and the multiplications respect the order.

Since $\ol D$ is finite if and only if $D$ is finite, it will suffice to show that, for a fixed $b\in B$, the set $D_b = \{\rho(a_1,\dots,a_n, b) \mid \rho\in R_\L,\ a_1,\ldots, a_n\in P\}$ is finite. Since $\varphi$ is order-preserving, the inverse image of an order-ideal is an order-ideal. Further, since $\leqc$ is a dual well-order by Corollary~\ref{c:dualwellorder}, $\varphi^{-1}[\da b] = \da U_b$, for some finite antichain $U_b\subseteq F$. Fix $a_1,\dots, a_n\in P$ and let $t_1,\ldots, t_n\in F$ such that $\varphi(t_i) = a_i$, for $i = 1, \dots, n$. Fix also $\rho(x_1,\dots,x_n,z)\in R_\L$ and let $\lambda(x_1, \ldots, x_n, y)$ be its corresponding multiplicative term given by Lemma~\ref{lem:backandforth}. Therefore, 
\begin{align*}
s\in \varphi^{-1}[\da\rho^{\Low(\m P)} (a_1,\dots,a_n, b)] &\iff \varphi(s)\leq \rho^{\Low(\m P)} (a_1,\dots,a_n, b) \\
 &\iff \lambda^{\Low(\m P)} (a_1\ldots, a_n, \varphi(s))\leq b\\
 &\iff\varphi(\lambda^{\m F} (t_1\ldots, t_n, s))\leq b \\
 &\iff \lambda^{\m F} (t_1\ldots, t_n, s)\in \varphi^{-1}[\da b]\\ 
 &\iff \lambda^{\m F} (t_1\ldots, t_n, s)\leqc u, \text{ for some } u\in U_b \\
 & \iff s\leqc \rho^{\m F} (t_1,\dots,t_n, u), \text{ for some } u\in U_b 
\end{align*}
Thus, we have shown that given $a_1,\dots,a_n\in P$, there exist $t_1,\dots,t_n \in F$ such that
\[
\varphi^{-1}[\da\rho^{\Low(\m P)} (a_1,\dots,a_n, b)] = \bigcup_{u\in U_b}\da\rho^{\m F} (t_1,\dots,t_n, u).
\]
Now, for every $u\in U_b$, the integrality of $\m F$ implies that $\lambda^\m F (t_1,\dots,t_n,u)\leqc \lambda^\m F (\ut,\dots,\ut,u) = u$, and thus $u\leqc \rho^{\m F} (t_1,\dots,t_n, u)$. Thus, $\rho^{\m F} (t_1,\dots,t_n, u)$ belongs to $\ua U_b = \{t\in F \mid {u\leqc t,} \text{ for some } {u\in U_b}\}$, which is a finite set because $U_b$ is finite and $\leqc$ is a dual partial well-order. It follows that there are only finitely many inverse images of the form $\varphi^{-1}[\da\rho^{\Low(\m P)} (a_1,\dots,a_n, b)]$ as $a_1,\dots,a_n$ range over all the elements of $P$. Also, since $\varphi[F] = P$, $\rho^{\Low(\m P)} (a_1,\dots,a_n, b)\neq \rho^{\Low(\m P)} (c_1,\dots,c_n, b)$ implies $\varphi^{-1}[\da\rho^{\Low(\m P)} (a_1,\dots,a_n, b)]\neq \varphi^{-1}[\da\rho^{\Low(\m P)}  (c_1,\dots,c_n, b)]$. These facts demonstrate the finiteness of $D_b$.
\end{proof}

Combining Lemmas~\ref{lem:embedding} and~\ref{l:finiteness}, we obtain the main result of this section:

\begin{theorem}\label{thm:HRL:FEP}
The variety $\cls{HRL}$ has the finite embeddability property.
\end{theorem}

As an application, we present a simple proof of the FEP for the variety of distributive integral residuated lattices, which was independently proved in~\cite{Bus11} and~\cite{GaJi}.

\begin{corollary}\label{DIRL:FEP}
The variety of distributive integral residuated lattices has the finite embeddability property.
\end{corollary}

\begin{proof}
Suppose that $\m B$ is a partial subalgebra of a distributive integral residuated lattice $\m A$. The ideal completion $\m{I(A)}$ of $\m A$ is a Heyting residuated lattice that has $\m A$ as a residuated lattice subreduct. To see this, observe that Corollary~\ref{completions-heyting} can be used to show that $\m{I(A)}$ is a nucleus-system of $\Low(\m{A})$. Further, as $\m{I(A)}$ is an algebraic distributive lattice, it possesses a Heyting implication. Thus, $\m B$ is a partial subalgebra of $\m{I(A)}$, and hence there is a finite Heyting residuated lattice $\m F$ that includes $\m B$ as a subalgebra. The algebra $\m F$, being a  a Heyting residuated lattice, is distributive.
\end{proof}

The method outlined in this section proves the finite embeddability property for other classes of residuated lattices, for example, commutative integral residuated lattices and  integral residuated lattices, which are the main objects of investigation of the articles~\cite{BvA02} and~\cite{BvA05}. The construction of the finite extension follows the proof of Lemmas~\ref{lem:embedding}, but it is simpler as it does not involve the Heyting implication.

The same approach shows that the variety $\cls{S}em\cls{IRL}$ of semilinear integral residuated lattices, namely the variety generated by integral residuated chains, satisfies the finite embeddability property.
The preceding method shows that a finite partial subalgebra $\m B$ of an integral residuated chain $\m A$ is embeddable into a finite integral chain. Indeed, the $\{\cdot,\ut\}$-subreduct $\m P$ of $\m A$ generated by $B$ is a totally ordered monoid (refer to the proof of Lemma~\ref{lem:embedding}), and hence $\Low(\m P)$ is totally ordered. But then the algebra $\m{\ol D}$  is an integral residuated chain.
The finiteness of $\m{\ol D}$ follows from Lemma~\ref{l:finiteness}.
This implies that any finite partial subalgebra of a semilinear integral residuated lattice is embeddable in a finite product of finite integral chains, and so the subalgebra it generates is finite.

Another interesting application of the preceding method is the proof of the finite embeddability property for the variety $\cls{I}nv\cls{IRL}$ of involutive integral residuated lattices, first stablished in~\cite{Wil06} and~\cite{GJKO07}.

\begin{theorem}\label{thm:INVIRL:FEP}
The variety $\cls{I}nv\cls{IRL}$ of involutive integral residuated lattices has the finite embeddability property.
\end{theorem}

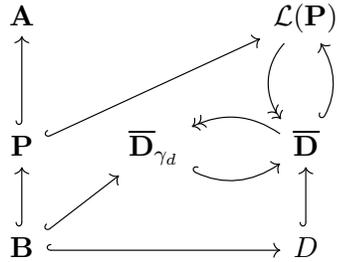
\begin{proof}\text{}

\begin{minipage}{0.1\textwidth}
\begin{tikzcd}
 \m A && \Low(\m P)\ar[->>, bend right=40, d]\\[2ex]
 \m P\ar[hook, u]\ar[hook, urr] & \ol{\m D}_{\gamma_d}\ar[bend right, hook, r] & \ol{\m D}\ar[hook, bend right, u]\ar[->>, bend right, l]\\
\m B\ar[hook, u]\ar[hook, ur]\ar[hook, rr] && D\ar[hook, u]
\end{tikzcd}

\end{minipage}\hfill
\begin{minipage}{0.57\textwidth}
Let $\m A$ be an involutive integral residuated lattice with a cyclic dualizing element $d$. Let $\m B$ be a finite full partial subalgebra of $\m A$. We can assume, without loss of generality, that $d,\ut\in B$ and that $b\ras_{\m A} d=b\ld_{\m A} d=d\rd_{\m A} b\in B$, for each $b\in B$. The fact that $d$ is a cyclic dualizing element implies that every element of $B$ is of the form $b\ras_{\m A} d$, for some $b\in B$; equivalently, $B\subseteq A_{\gamma_d}$ (see Lemma~\ref{lem:specialnuclei}). Let $\m P$ be the $\{\cdot,\ut\}$-subreduct of $\m A$ generated by  $B$. We
\end{minipage}

\noindent  again use $\m P$ to denote the full partial subalgebra of $\m A$ on $P$.   We know, in view of the preceding assumptions on $\m B$ and Lemma~\ref{residuals}, that $b\ras_{\m B} d=b\ras_{\m A} d=b\ras_{\m P} d=b\ras_{\Low(\m P)} d\in B\subseteq P$. We claim that $d$ is a cyclic element of $\Low(\m P)$ (with respect to the residuals of $\Low(\m P)$). A word of caution is necessary here. While $d$ is a cyclic element of $\m A$, there is no guarantee that the residuals $x\ld_{\m A} d$ and $d\rd_{\m A} x$ are in $P$, for $x\in P$. We claim, however, that  $x\ld_{\Low(\m P)} d=d\rd_{\Low(\m P)} x$, for all $x\in P$. Indeed, given an arbitrary element $x\in P$, there exist elements $x_1, \dots, x_n\in B$ such that $x=x_1\dots x_n.$ We prove inductively that $x\ld_{\Low(\m P)} d=d\rd_{\Low(\m P)} x$. To simplify the notation in the computation below, we will use $\ld$ and $\rd$ in the place of  $\ld_{\Low(\m P)}$ and $\rd_{\Low(\m P)}$, respectively. Setting $y=x_1\dots x_{n-1},$ assume  that  $y\ld d=d\rd y$.  Then $x\ld d=yx_n\ld d=x_n\ld(y\ld d)=x_n\ld(d\rd y)=(x_n\ld d)\rd y=(d\rd x_n)\rd y=d\rd yx_n=d\rd x.$ Finally, if $z$ is an arbitrary element of $\Low(\m P)$, there exists a subset $W$ of $P$ such that $z=\Join[{\Low(\m P)}]W$.  Hence, $z\ld d=\Meet\{w\ld d \mid w\in W\}=\Meet\{d\rd w \mid w\in W\}=d\rd z.$ We have shown that $d$ is a cyclic element of $\Low(\m P)$.

As $\ol{\m D}$ is a nucleus system of $\Low(\m P)$ (see Lemma~\ref{lem:embedding}), $d$ is a cyclic element of $\ol{\m D}$ (see Lemma~\ref{lem:nuclear:system}). Consider the nucleus $\gamma_d$ on $\ol{\m D}$. By Lemma~\ref{lem:involuted1}, 
$\ol{\m D}_{\gamma_d}$ is an involutive integral residuated lattice. Now $\ol{\m D}$ is finite by Lemma~\ref{l:finiteness}, and hence so is $\ol{\m D}_{\gamma_d}$. What is left to observe is that every element of $B$ is fixed by $\gamma_d$, and hence the inclusion $\m B\mapsto \ol{\m D}_{\gamma_d}$ is an embedding.
\end{proof}

\begin{corollary}\label{thm:INVCIRL:FEP}
The variety $\cls{I}nv\cls{CIRL}$ of involutive, commutative, integral residuated lattices has the finite embeddability property.
\end{corollary}
\begin{proof}
The construction of Lemma~\ref{lem:embedding} preserves commutativity.
\end{proof}

\section{A general discussion of the finite embeddability property and its implications.}
\label{s:fep}

The finite embeddability property (FEP) for general algebras was first introduced and  studied systematically by T. Evans (see~\cite{Eva51, Eva53, Eva69}).  Additional relevant references include \cite{BaNe72, Boo59, Coh81, Dek95} and the ones listed in Section~\ref{sec:FEP:HRL}. 

The aim of this section is to provide a survey of FEP by clarifying relationships among several related notions and reviewing general theorems with detailed proofs that remedy some gaps in the literature. For notions not defined here, we refer the reader to~\cite{BS81}, ~\cite{Coh81}, ~\cite{Jez08}, or~\cite{Bur86}.\footnote{An electronic version of this book can be found in~\url{http://www.mathematik.tu-darmstadt.de/Math-Net/Lehrveranstaltungen/Lehrmaterial/SS2002/AllgemeineAlgebra/}.}

\subsection{Preliminaries}

For any set of variables\footnote{If necessary, we can assume that we have an infinite supply of variables. Formally, we fix a class $\X$ containing \emph{all the variables}, and thus, a set of variables is just a subset $X\subseteq\X$.} $X$, let ${\m T}(X)$ be the term algebra on the language $\L$ with variables in $X$. By an \emph{equation} (on the language $\L$ with variables in $X$) we mean an ordered pair $\pair{t, t'}$ of terms of $T(X)$, written as $t\eq t'$, and a \emph{quasi-equation} is a formula of the form
\[
{t_1\eq t'_1} \conj \dots \conj {t_m\eq t'_m} \Ra t\eq t'
\]
where $t\eq t'$ and $t_{i}\eq t'_{i}$ are equations for $i = 1,\dots,m$.

For any partial $\L$-algebra $\m P$, an \emph{assignment} of $X$ in $\m P$ is a map $v \colon X \to P$. Any such assignment $v$ can be extended uniquely to a \emph{valuation} $\tilde{v}$, which is a partial map\footnote{For every assignment $v\colon X\to P$ on a partial algebra $\m P$, the map $\tilde v$ is a homomorphism from the full partial subalgebra $\m T(X)\rest\dom(\tilde v)$ to $\m P$, and it is the \emph{largest one} satisfying that $\tilde v(x) = v(x)$, for every $x\in X$.} on $T(X)$ defined recursively as follows:
\begin{itemize}
\item For each variable $x\in X$, $\tilde{v}(x) = v(x)$.
\item If $f\in L$ is an operation symbol, and if $t_1,\dots,t_{\tau{(f)}}$ are terms of $T(X)$, for which $\tilde{v}{(t_{i})}$ is defined, say $\tilde{v}(t_{i}) = p_{i}$ $(1\leq i\leq \tau(f))$, and such that $f^{\m P}(p_1,\dots,p_{\tau(f)})$ is also defined, then we define
\[
\tilde{v}(f(t_1,\dots,t_{\tau(f)}))=f^{\m P}(p_1,\dots,p_{\tau(f)}).
\]
Otherwise, it is undefined.
\end{itemize}

A partial $\L$-algebra $\m P$ \emph{satisfies} an equation $t\eq t'$ with respect to an assignment $v \colon X \to P$,
in symbols ${\m P}\models_v {t\eq t'}$, if $\tilde v(t)$ and $\tilde v(t')$ are defined and $\tilde v(t) = \tilde v(t')$. We say that $\m P$ \emph{satisfies} the equation $t\eq t'$ if $\m P$ satisfies it with respect to every assignment. A class of algebras $\cls K$ satisfies an equation if every algebra in $\cls K$ satisfies it. A partial $\L$-algebra $\m P$ \emph{satisfies} a quasi-equation $q = {t_1\eq t'_1} \conj \dots \conj {t_m\eq t'_m} \Ra t\eq t'$ with respect to an assignment $v \colon X \to P$, in symbols ${\m P}\models_v q$, if ${\m P}\models_v {t\eq t'}$, whenever ${\m P}\models_v {t_i\eq t'_i}$, for every $i = 1,\dots, m$. The partial algebra $\m P$ \emph{satisfies} a quasi-equation if $\m P$ satisfies it with respect to every assignment. A class of algebras $\cls K$ satisfies a quasi-equation if every algebra in $\cls K$ satisfies it. 

A \emph{variety} or \emph{equational class} is a class of $\L$-algebras defined by a set of equations.  Analogously, a \emph{quasi-variety} is a class of $\L$-algebras defined by a set of quasi-equations. If $\cls K$ is a class of $\L$-algebras, the \emph{(quasi\nbd-)variety generated} by $\cls K$ is the class $\mathbb V(\cls K)$ (resp.~$\mathbb Q(\cls K)$) of all the algebras satisfying all the (quasi\nbd-)equations that are satisfied by the members of $\cls K$.

\subsection{Finitely presented algebras}
\label{subsection:finitely-presented-algebras}

Given a variety $\cls{V}$ of $\L$-algebras and a nonempty set $X$, we denote by $\m{F}_{\cls{V}}(X)$, or simply $\m{F}(X)$, the  $\cls{V}$-\emph{free algebra over $X$}. The homomorphism $\varphi_{\cls{V}}^{X}\colon \m{T} (X) \to \m{F}_{\cls{V}}(X)$ that extends the identity on $X$ will play an important role in the ensuing considerations. We will denote its value at $t \in T (X)$ by $\bar{t}$, that is, $\varphi_{\cls{V}}^X(t)=\bar{t}$, and likewise write $\bar{\Sigma} = \{\pair{\bar t , \bar s} \mid t\eq s \in \Sigma\}$ for a set $\Sigma$ of equations  with variables in $X$. The \emph{congruence lattice} of an algebra $\m{A}$ will be denoted by $\Con(\m{A})$. For $S \subseteq A^2$, we write $\cg{\m{A}} (S)$ for the congruence relation on $\m{A}$ generated by $S$, abbreviating to $\cg{\m{A}}(a,b)$ for the principal congruence on $\m{A}$ generated by a pair $\pair{a,b} \in A^2$. For $\theta\in \Con(\m{A})$  and $a\in A$, we denote the equivalence class of $a$ relative to $\theta$ by $[a]_\theta$ or simply $[a]$.

Let $\cls V$ be a variety, $X$ an arbitrary set (of variables), and  $\Sigma$ a set of equations in these variables. An algebra $\m A\in \cls V$ is said to be defined by \emph{generators} $X$ and \emph{relations} $\Sigma$, and write $\m A=\cls V\pair{X\mid \Sigma}$ or simply $\m A=\pair{X\mid \Sigma}$, in case $\m A=\F_{\cls V}(X) / \cg{\F_{\cls V}(X)} (\bar \Sigma)$. In view of the preceding discussion, $\m A\cong{{\m T}(X)/ [\theta_{\cls V}^{X}} \jn {\cg{\m T(X)} (\Sigma)}]$, where $\theta_{\cls V}^X$ is the kernel of the aforementioned homomorphism $\varphi_{\cls V}^X$. We refer to $\cls V\pair{X\mid \Sigma}$ as a \emph{presentation} of $\m A$. The algebra $\m A$ is called \emph{finitely presented} provided $X$ and $\Sigma$ are finite. Thus,  $\m A\in \cls V$ is finitely presented if and only if it is the quotient algebra of a finitely generated $\cls V$-free algebra by a compact congruence.
An algebra $\m A$ is (\emph{finitely}) \emph{presentable} if it is isomorphic to a (finitely) presented algebra. As usual, we will just write $\pair{X\mid \Sigma}$ for $\cls V\pair{X\mid \Sigma}$ if $\cls V$ is clear from the context. 

Given two sets of equations $\Sigma,\Delta$ in the set of variables $X$, we say that $\Sigma$ implies $\Delta$ in the variety ${\cls V}$, and write $\Sigma\models_{\cls V} \Delta$, if for every algebra $\m A$ in ${\cls V}$ and every homomorphism $\varphi \colon {\m T}(X) \to {\m A}$, $\Delta\subseteq\ker\varphi$ whenever $\Sigma\subseteq\ker\varphi$. It can be readily seen that $\models_{\cls V}$ is a structural consequence relation (if $\Delta\subseteq \Sigma$ then $\Sigma\models_{\cls V} \Delta$, it is transitive, and for every substitution $\sigma\in\End({\m T}(X))$, $\Sigma\models_{\cls V}\Delta$ implies $\sigma[\Sigma]\models_{\cls V}\sigma[\Delta]$), and it can be characterized in the following way (see~\cite{MMT14}).

\begin{lemma}
If ${\cls V}$ is a variety, $X$ is a set of variables and $\Sigma,\Delta\subseteq{\m T}(X)^2$, then the following conditions are equivalent:
\begin{enumerate}[(i)]
\item $\Sigma\models_{\cls V}\Delta$,
\item $\Delta\subseteq \theta_{\cls V}^{X}\jn \cg{{\m T}(X)} (\Sigma)$,
\item $\bar\Delta\subseteq\cg{\F_{\cls V}(X)}(\bar\Sigma)$.
\end{enumerate}
\end{lemma}
\noindent In what follows, we drop the subscript of $\models_{\cls V}$ whenever there is no danger of confusion. 

The next two lemmas are slight modifications of ~\cite[Theorem III.8.4]{Coh81}. They relate two different presentations of an algebra and describe a process for obtaining each of the presentations from the other. Corollary~\ref{cor:flat:presentation} provides a typical application of these results.

\begin{lemma}\label{Cohn:operations:one}
Let ${\cls V}$ be a variety and let $\pair{X\mid \Sigma}$ be a presentation of an algebra $\m A\in\cls V$. Let $Z$ and $\Gamma$ be sets obtained from $X$ and $\Sigma$ by applying the operations below or their inverses:
\begin{enumerate}[(i)]
\item If $\Delta$ is any set of equations with variables in $X$ such that $\Sigma\models \Delta$, set $Z=X$ and $\Gamma=\Sigma\cup \Delta$.
\item If $Y$ is a set disjoint from $X$ and $\alpha \colon Y \to T(X)$ is any map, set $Z=X\cup Y$ and $\Gamma=\Sigma\cup\{\pair{y,\alpha(y)} \mid y\in Y\}$.
\end{enumerate}
Then $\pair{Z\mid \Gamma}$ is an alternative presentation of $\m A$.
\end{lemma}

\begin{proof}
Let ${\m A} = \pair{X\mid \Sigma}$ and let $\Delta$ be a set equations such that $\Sigma\models \Delta$. Then $\bar \Delta\subseteq \cg{\F(X)}(\bar \Sigma)$, and hence
\[
\pair{X\mid \Sigma} = \F(X)/\cg{\F(X)}(\bar \Sigma) = \F(X)/\cg{\F(X)}(\bar \Sigma\cup\bar \Delta) = \pair{X\mid \Sigma\cup \Delta}.
\]

To prove (ii), consider a set $Y$ of variables disjoint from $X$ and let $\alpha \colon Y \to T(X)$ be any map. Set $Z=X\cup Y$ and define a surjective homomorphism $\rho \colon {\m T}(Z) \to {\m T}(X)$ such that $\rho(x) = x$, for every $x\in X$ and $\rho(y) = \alpha(y)$, for every $y\in Y$. Let $\widetilde\rho \colon \F(Z) \to \F(X)$ be the unique homomorphism determined by $\widetilde\rho(\bar z) = \overline{\rho(z)}$, for every $z\in Z$. 
{
Therefore, we obtain the solid part of the diagram below, where  $T_\alpha = \{\pair{y,\alpha(y)} \mid y\in Y\}$, $\Gamma=\Sigma\cup T_\alpha$, $\pi_X\colon \m F(X)\to\pair{X\mid \Sigma}$ and $\pi_Z\colon \m F(Z)\to\pair{Z\mid \Gamma}$ are the canonical projections, and $\pi$ is the composition $\pi = \pi_{X}\widetilde\rho$. Consider also $\theta_X=\cg{\F(X)}(\bar \Sigma)$, and $\theta_Z = \cg{\F(Z)}{(\bar\Sigma\cup\bar T_\alpha})$.} To prove that  $\pair{X\mid \Sigma}\cong\pair{Z \mid \Gamma}$, it will suffice to show that $\ker\pi = \theta_Z$.
\[
\begin{tikzcd}
{\m T(X)}\ar[d] & {\m T(Z)}\ar[d]\ar[l, "\rho"'] \\
{\F(X)}\ar[d, two heads, "\pi_{X}"'] & {\F(Z)}\ar[l, two heads, "\widetilde\rho"'] \ar[d, two heads, "\pi_{Z}"] \ar[dl, two heads, "\pi"']\\
{\pair{X\mid \Sigma}}\ar[r, dashed, "i"]	   & 	{\pair{Z \mid \Gamma}}
\end{tikzcd}
\]

Note that $\ker\pi = \ker(\pi_{X}\widetilde\rho\,) = \widetilde\rho\,^{-1}[\theta_{X}]$. Further, $\theta_{Z}\subseteq \widetilde\rho\,^{-1}[\theta_{X}]$ if and only if $\bar\Gamma=\bar \Sigma\cup \bar T_\alpha \subseteq \widetilde\rho\,^{-1}[\theta_{X}]$ if and only if $\widetilde\rho\,[\bar\Gamma]\subseteq\theta_{X}$. The latter condition is true. Indeed, by the definition of $\widetilde\rho$, we have that $\widetilde\rho\,[\bar \Sigma] = \bar \Sigma\subseteq \theta_{X}$ and $\widetilde\rho\,[\bar T_\alpha]$ is a subset of the identity congruence of $\F(X)$.
 
 For the reverse inclusion, suppose that $t\eq s$ is an equation such that $\pair{\bar t,\bar s} \in\ker\pi$. Then $\pair{\widetilde\rho(\bar t\,),\widetilde\rho(\bar s)}\in\theta_{X}$, and therefore $\pair{\widetilde\rho(\bar t\,),\widetilde\rho(\bar s)}\in\theta_{Z}$, since $\F(X)$ is a subalgebra of  $\F(Z)$. As $\rho$ is a homomorphism satisfying that $\bar T_\alpha\subseteq\theta_{Z}$ and for every $z\in Z$, $\pair{z,\rho(z)}\in T_\alpha$, it follows that for every $w\in F(Z)$, $\pair{w,\widetilde\rho(w)}\in\theta_{Z}$. In particular, $\pair{\bar t,\widetilde\rho(\bar t\,)}, \pair{\bar s,\widetilde\rho(\bar s)}\in\theta_{Z}$, showing that $\pair{\bar t,\bar s}\in\theta_{Z}$. We have shown that 
 $\ker\pi = \theta_Z$. Hence, the general homomorphism theorem implies that there is a unique isomorphism $i\colon \pair{X\mid \Sigma} \to \pair{Z \mid \Gamma}$ such that $i\pi = \pi_Z$, as was to be shown.
\end{proof}

\begin{remark}\label{rem:isomorphism}
We note for future reference that the aforementioned isomorphism $i\colon \pair{X\mid \Sigma} \to \pair{Z \mid \Gamma}$ satisfies  $i([x]_{\theta_X}) = [x]_{\theta_Z}$, for every $x\in X$. This implies that $i([\bar t\,]_{\theta_X}) = [\bar t\,]_{\theta_Z}$ for every $t\in T(X)$.
\end{remark}

\begin{lemma}\label{Cohn:operations:two}
Let ${\cls V}$ be a variety. Given two presentations $\pair{X\mid \Sigma}$ and $\pair{Y\mid \Delta}$ of isomorphic algebras in ${\cls V}$, each can be obtained from the other by applying operations of type (i) and (ii) and their inverses, as described in Lemma~\ref{Cohn:operations:one}.
\end{lemma}

\begin{proof}
Let $\theta_{X}=\cg{\F(X)} (\bar \Sigma)$, $\theta_{Y}=\cg{\F(Y)} (\bar \Delta)$, and $Z=X\cup Y$. Without loss of generality, we may assume  that $X\cap Y = \emptyset$. The definition of a presentation  yields $\pair{X\mid \Sigma}=\F(X) / \theta_{X}$ and $\pair{Y\mid \Delta}=\F(Y) / \theta_{Y}$. 

Let us start with an isomorphism $\varphi \colon \pair{Y\mid \Delta} \to \pair{X\mid \Sigma}$.   For every $y\in Y$, there exists a term $\alpha(y)\in T(X)$ such that $\varphi([\bar{y}]_{\theta_{Y}}) = \big[\,\overline{\alpha(y)}\,\big]_{\theta_{X}}$. Let $T_\alpha = \{\pair{y,\alpha(y)} \mid y\in Y\}$ and $\theta_{Z}= \cg{\F(Z)} (\bar \Sigma\cup \bar{T}_\alpha)$. In view of Lemma~\ref{Cohn:operations:one}, there is an isomorphism $i \colon \pair{X\mid \Sigma} \to \pair{Z\mid \Sigma\cup T_\alpha}=\F(Z) / \theta_{Z}$. By Remark~\ref{rem:isomorphism}, it follows in particular that $i$ maps $[\bar t\,]_{\theta_{X}}$ to $[\bar t\,]_{\theta_{Z}}$, for every $t\in T(X)$.
 Now for every $y \in Y$, $i\varphi([\bar y]_{\theta_{Y}}) = i \big(\big[\,\ol{\alpha(y)}\,\big]_{\theta_{X}}\big) = \big[\,\ol{\alpha(y)}\,\big]_{\theta_{Z}} = [\bar y]_{\theta_Z}$, because of the choice of the map $\alpha$ and the definition of the set $T_\alpha$. Hence, $i\varphi([\bar t\,]_{\theta_Y}) = [\bar t\,]_{\theta_Z}$, for every $t\in T(Y)$. Therefore, if $\pair{t,s}\in \Delta$, then $[\bar s]_{\theta_{Z}} = i\varphi([\bar s]_{\theta_{Y}}) = i\varphi([\bar t\,]_{\theta_{Y}}) = [\bar t\,]_{\theta_{Z}}$, which means that $\Sigma\cup T_\alpha\models \Delta$. Thus, again by Lemma~\ref{Cohn:operations:one}, $\pair{Z\mid \Sigma\cup T_\alpha} \cong \pair{Z\mid \Sigma\cup T_\alpha\cup \Delta}$.

Lastly, for every $x\in X$, there exists $\beta(x)\in T(Y)$ such that $\big[\,\overline{\beta(x)}\,\big]_{\theta_{Y}} = \varphi^{-1}([\bar x]_{\theta_{X}})$. Set $T_\beta = \{\pair{x,\beta(x)} \mid x\in X\}$. We have for $x\in X$, $[\bar x]_{\theta_{Z}} = i([\bar x]_{\theta_{X}}) = i\varphi\big(\big[\,\overline{\beta(x)}\,\big]_{\theta_{Y}}\big) = \big[\,\overline{\beta(x)}\,\big]_{\theta_{Z}}$, that is $\Sigma\cup T_\alpha\models T_\beta$, and hence also $\Sigma\cup T_\alpha\cup \Delta\models T_\beta$. This yields $ \pair{Z\mid \Sigma\cup T_\alpha\cup \Delta} \cong \pair{Z\mid \Sigma\cup T_\alpha\cup \Delta\cup T_\beta}$.

The preceding considerations demonstrate that $\pair{Z\mid \Sigma\cup T_\alpha\cup \Delta\cup T_\beta}$ is derived from $\pair{X\mid \Sigma}$ with the use of operations of  type~(i) and~(ii). Analogously,  $\pair{Z\mid \Sigma\cup T_\alpha\cup \Delta\cup T_\beta}$ is derived from $\pair{Y\mid \Delta}$ with the use of the same operations. In conclusion, $\pair{Y\mid \Delta}$ is derived from $\pair{X\mid \Sigma}$ with the use of type~(i) and~(ii) operations and their inverses, and vice versa.
\end{proof}

\begin{definition}
A presentation $\pair{X\mid \Sigma}$ of an algebra is said to be \emph{flat} if the equations in $\Sigma$ are of the form $f(x_1,\dots,x_{\tau(f)})\eq x$, for operation symbols $f\in L$ and $x_1,\dots,x_{\tau(f)}, x\in X$. 
\end{definition}

We prove below that any presentable algebra admits a flat presentation. 

\begin{corollary}\label{cor:flat:presentation}
Every finitely presentable algebra in a finite language admits a finite flat presentation.
\end{corollary}

\begin{proof}
Let $\m A=\pair{X\mid \Sigma}$ be a finitely presented $\L$-algebra and let $S$ be the set of all subterms appearing in the equations of $\Sigma$. Given $w\in S$, set $\hat w = w$, if $w$ is a constant of $\L$ or $w\in X$; otherwise, let $\hat w$ be a new variable. That is to say, the set $Y = \{\hat w \mid w\in S,\ w \text{ is not a constant and } w\not\in X\}$ is disjoint from $X$ and $\hat w \neq\hat w'$ if $w\neq w'$. Let $\alpha \colon Y \to T(X)$ be defined by $\alpha(\hat w) = w$, and $T_\alpha = \{\pair{\hat w, w} \mid \hat w\in Y\}$. In view of Lemma~\ref{Cohn:operations:one}, $\pair{X\mid \Sigma} \cong \pair{X\cup Y\mid \Sigma\cup T_\alpha}$. Consider $\Sigma_1 = \{f(\hat w_1,\dots,\hat w_{\tau(f)})\eq\hat w \mid w = f(w_1,\dots, w_{\tau(f)}) \in S\}$ and $\Sigma_2 = \{\hat t \eq \hat r \mid t\eq r \in \Sigma\}$. It is easy to see that for every $w = f(w_1,\dots, w_{\tau(f)})\in S$, $T_\alpha\models f(\hat w_1,\dots,\hat w_{\tau(f)})\eq\hat w$, and that $\Sigma_1\models \hat w\eq w$, by induction in the complexity of $w$. This shows that $T_\alpha\models \Sigma_1$, and conversely $\Sigma_1\models T_\alpha$. But then $\Sigma\cup T_\alpha \models \Sigma_1\cup \Sigma_2$ and $\Sigma_1\cup \Sigma_2\models \Sigma\cup T_\alpha$. Hence, we obtain
\[
\pair{X\cup Y\mid \Sigma} \cong \pair{X\cup Y\mid \Sigma\cup T_\alpha} \cong \pair{X\cup Y\mid \Sigma\cup T_\alpha\cup \Sigma_1\cup \Sigma_2} \cong \pair{X\cup Y\mid \Sigma_1\cup \Sigma_2}.
\]
Finally, let $Z=X\cup Y$ and $\Gamma = \Sigma_1\cup \Sigma_2$. If  $\Gamma$ contains an equation of the form $z_1\eq z_2$, with $z_1,z_2\in Z$, we remove $z_2$ from $Z$ and $z_1\eq z_2$ from $\Gamma$, and we substitute uniformly all occurrences of the variable $z_2$ in the terms of $\Gamma$ by $z_1$. We repeat the process until we obtain a flat presentation $\pair{Z\mid \Gamma}$.
\end{proof}

\subsection{The finite embeddability property and the strong finite model property}

In this section we introduce two semantic properties of classes of algebras, the Finite Model Property and the Strong Finite Model Property, and study their relationship with the FEP.

\begin{definition}\label{def:FMP}\label{def:SFMP}
A class of algebras $\cls K$ is said to have the \emph{finite model property} (FMP, for short) if every equation that fails in $\cls K$ fails in a finite member of $\cls K$. The class $\cls K$ is said to have the \emph{strong finite model property} (SFMP, for short) if every quasi-equation that fails in $\cls K$ fails in a finite member of $\cls K$.
\end{definition}

It is easy to deduce directly by the definitions that $\cls K$ has the FMP if and only if $\cls K\subseteq \mathbb V(\cls K_F)$, where $\cls K_F$ denotes the class of finite algebras in $\cls K$. Likewise, $\cls K$ has the SFMP if and only if $\cls K\subseteq \mathbb Q(\cls K_F)$. Thus, in the case that $\cls V$ is a variety, $\cls V$ has the FMP if and only if it is generated (as a variety) by its finite members, and likewise, a quasi-variety $\cls Q$ has the SFMP if and only if it is generated (as a quasi-variety) by its finite members.

Theorem~\ref{thm:FEP:SFMP} below describes the relationship of the FEP and the SFMP. We start by stating and proving two technical lemmas. Recall that every assignment $v\colon X\to P$, where $\m P$ is a partial algebra, can be extended uniquely to a valuation $\tilde v$.

\begin{lemma}\label{valuations:and:homomorph}
If $\varphi \colon {\m P} \to {\m Q}$ is a homomorphism of partial algebras, $v \colon X \to P$ is an assignment, and $u = \varphi v\colon X\to Q$, then for every term $s$ for which $\tilde v(s)$ is defined, $\tilde u(s)$ is also defined in $\m Q$ and $\tilde u(s) = \varphi(\tilde v(s))$.
\end{lemma}

\begin{proof}
If $s$ is a variable, then it is evident. Otherwise, suppose that $s = f(t_1,\dots,t_{\tau(f)})$ and $\tilde v(s)$ is defined. Hence $\tilde v (t_1)$, \dots, $\tilde v(t_{\tau(f)})$, and $f^{\m P}(\tilde v(t_1),$ $\dots, \tilde v(t_{\tau(f)}))$ are defined and, by the induction hypothesis, $\tilde u(t_i) = \varphi (\tilde v(t_i))$, for every $i = 1,\dots, \tau(f)$. Since $\varphi$ is a homomorphism, $f^{\m Q}(\varphi(\tilde v(t_1)),\dots,$ $\varphi(\tilde v(t_{\tau(f)})))$ is defined and this implies that $\tilde u(s)$ is also defined and
\begin{align*}
\tilde u(s) &= f^{\m Q}(\tilde u(t_1),\dots,\tilde u(t_{\tau(f)})) 
  = f^{\m Q}(\varphi(\tilde v(t_1)),\dots,\varphi(\tilde v(t_{\tau(f)}))) \\
&= \varphi(f^{\m P}(\tilde v(t_1),\dots, \tilde v(t_{\tau(f)}))) 
  = \varphi(\tilde v(s)).\qedhere
\end{align*}
\end{proof}

Given a family $\{{\m Q}_i \mid i\in I\}$ of partial $\L$-algebras, the \emph{direct product} $\m Q = \prod_I {\m Q}_i$ is defined as usual: for every $f\in L$ and $a_1,\dots, a_{\tau(f)} \in \prod_I Q_i$,  we know that $f^{\m Q}(a_1,\dots, a_{\tau(f)})$ is defined if and only if for every $i\in I$, $f^{\m Q_i}(a_1(i),\dots, a_{\tau(f)}(i))$ is defined, and, in this case, 
\[
f^{\m Q}(a_1,\dots, a_{\tau(f)})(i) = f^{{\m Q}_i}(a_1(i),\dots, a_{\tau(f)}(i)).
\]
It can be readily proven that the projection maps $\pi_j \colon \prod_I {\m Q}_i \to {\m Q}_j$ are homomorphisms, and that $\prod_I {\m Q}_i$ satisfies the universal property of the product in the class of all partial $\L$-algebras, which is stated in the next lemma. The proof is straightforward.

\begin{lemma}
If $\{\varphi_i \colon {\m P} \to {\m Q}_i \mid i\in I\}$ is a family of homomorphisms of partial $\L$-algebras, then there exists a unique homomorphism $\varphi \colon {\m P} \to \prod_I {\m Q}_i$ satisfying $\varphi_i = \pi_i \varphi$, for all $i\in I$.
\end{lemma}

\begin{definition}
Given a finite partial algebra $\m P$, we will consider a fixed injective map $\hat\ \colon P \to X$ assigning one variable $\hat p$ to every element $p\in P$. The \emph{diagram} of $\m P$ (with respect to the map $\hat\ $\,) is the set $\Diag{\m P}$ defined as
\[
\{f(\hat p_1,\dots,\hat p_{\tau(f)}) \eq \hat p \mid f\in L,\ p_1,\dots, p_{\tau(f)},p\in P, \ f^{\m P}(p_1,\dots,p_{\tau(f)}) = p\}.
\]
\end{definition}
That is, the diagram of $\m P$ is a syntactic description of the partial algebra $\m P$, an ``operation table" for $\m P$. It is worth noticing that the usual definition of the diagram of an algebra uses ``new constants'' instead of variables. Nonetheless, we find that the use of variables fits better our purposes.

The next result, due to Blok and van Alten (see~\cite{BvA02}), states that the three aforementioned properties are equivalent under very general conditions.

\begin{theorem}\label{thm:FEP:SFMP}
Let $\cls K$ be a class of algebras of language $\L$. Then,  (i) implies~(ii) and~(ii) implies~(iii) below. If the language $\L$ is finite, then (ii) implies~(i). If $\cls K$ is closed under finite products of its finite members, then (iii) implies~(ii), making all three statements equivalent.
\begin{enumerate}[(i)]
\item $\cls K$ has the FEP.
\item $\cls K$ has the FEP\textss+.
\item $\cls K$ has the SFMP.
\end{enumerate}
\end{theorem}

\begin{proof}\
\begin{itemproof}
\item[(i)${}\iff{}$(ii):] This equivalence is an immediate consequence of Lemma~\ref{lem:FEP:FEPp}.

\item[(i)${}\Ra{}$(iii):] Suppose  that $\cls K$ has the FEP, and let $q$ be a quasi-equation that fails in $\cls K$. Let ${\m A}$ be an algebra in $\cls K$ and let $v$ be an assignment in $\m A$ that witnesses the failure of $q$. Let $B = \{\tilde v(w) \mid w \text{\ is a subterm appearing in\ } q\}$ and consider the full partial subalgebra ${\m P} = {\m A}\rest B$ of $\m A$. By assumption, there exist a finite algebra ${\m C}$ in $\cls K$ and an embedding $\varphi \colon {\m P} \to {\m C}$. We proceed to  show that $q$ fails in ${\m C}$, and therefore $\cls K$ has the SFMP. Without loss of generality, we may assume that $v \colon X \to B$, and therefore $u = \varphi v$ is an assignment in ${\m C}$. Suppose that $q = {t_1\eq t'_1} \conj \dots \conj {t_m\eq t'_m} \Ra t\eq t'$. Then, $\tilde v(t_i) = \tilde v(t'_i)$ for all $i$ but $\tilde v(t)\neq \tilde v(t')$, because ${\m A}\not\models_v q$. By Lemma~\ref{valuations:and:homomorph}, $\tilde u(t_i) = \varphi (\tilde v(t_i)) = \varphi (\tilde v(t'_i)) = \tilde u(t'_i)$, and by the injectivity of $\varphi$, $\tilde u(t) = \varphi (\tilde v(t)) \neq \varphi (\tilde v(t')) = \tilde u(t')$.

\item[(iii)${}\Ra{}$(i):] Suppose that $\cls K$ has the SFMP and is closed under finite products of its finite members. Let ${\m B}$ be a finite partial subalgebra of ${\m A}\in \cls K$. We fix a set of variables $X$ and an injective map $\hat\ \colon B \to X$ and consider the diagram $\Diag{\m B}$ of $\m B$, which is a finite set of equations because $\L$ is finite. Let $\&\Diag{\m B}$  denote the conjunction of all these equations. For any pair $b,b'$ of distinct elements in $B$, consider the following quasi-equation $q_{b,b'}$:
\[
\&\Diag{\m B} \Ra \hat b\eq \hat b'.
\]
Let $v \colon X \to A$ be any assignment in $\m A$ such that $v(\hat b) = b$, for all $b\in B$. We note that:
\begin{enumerate}[(a)]
\item ${\m A}\models_v\Diag{\m B}$. Indeed, if $f(\hat b_1,\dots, \hat b_{\tau(f)}) \eq \hat b$ is in $\Diag{\m B}$, then we have that $f^{\m B}(b_1,\dots, b_{\tau(f)}) = b$, and so
\begin{align*}
 \tilde v(f(\hat b_1,\dots, \hat b_{\tau(f)})) &= f^{\m A}(v(\hat b_1),\dots v(\hat b_{\tau(f)})) = f^{\m A}(b_1,\dots, b_{\tau(f)}) \\ 
  &= f^{\m B}(b_1,\dots, b_{\tau(f)}) = b = \tilde v (\hat b).
\end{align*}

\item ${\m A}\not\models_v \hat b\eq\hat b'$ for any two different elements $b,b'\in B$, because $v(\hat b) = b \neq b' = v(\hat b')$.
\end{enumerate}
Therefore, the quasi-equations $q_{b,b'}$ fail in $\m A\in \cls K$, for all $b\neq b'\in B$. 
Since $\cls K$ has the SFMP, it follows that\,---for each such pair  $b, b'$---\,there exists a finite algebra ${\m C}_{b,b'}\in \cls K$ in which $q_{b,b'}$ fails. Let ${\m C}$ be the direct product of the algebras ${\m C}_{b,b'}$ for all $b\neq b'\in B$. Obviously, $\m C$ is finite and as $\cls K$ is closed under finite products of its finite members, ${\m C}\in \cls K$.

We are going to define a homomorphism $\varphi \colon {\m B} \to {\m C}$ and prove that it is an embedding. First, for each $b\neq b'\in B$ consider an assignment\footnote{Of course, the assignment $w$ depends on $b$ and $b'$, although it is not reflected in the notation.} $w \colon X \to C_{b,b'}$ such that ${\m C}_{b,b'}\not\models_w q_{b,b'}$. We define $\varphi_{b,b'} \colon {\m B} \to {\m C}_{b,b'}$ by $\varphi_{b,b'}(a) = w(\hat a)$, for every $a\in B$. The map $\varphi_{b,b'}$ is a homomorphism. Indeed, if $f\in L$ and $b_1,\dots, b_{\tau(f)},a\in B$ are such that $f^{\m B}(b_1,\dots, b_{\tau(f)}) = a$, then $f(\hat b_1,\dots, \hat b_{\tau(f)})\eq \hat a$ is in $\Diag{\m B}$, and hence
\begin{align*}
 \varphi_{b,b'}(f^{\m B}(b_1,\dots, b_{\tau(f)})) &= \varphi_{b,b'}(a) = w(\hat a) = \tilde w(f(\hat b_1,\dots, \hat b_{\tau(f)}))
 \\ &= f^{{\m C}_{b,b'}}(\tilde w(\hat b_1),\dots, \tilde w(\hat b_{\tau(f)}))
 \\ &= f^{{\m C}_{b,b'}}(w(\hat b_1),\dots, w(\hat b_{\tau(f)}))
 \\ &= f^{{\m C}_{b,b'}}(\varphi_{b,b'} (b_1),\dots, \varphi_{b,b'}( b_{\tau(f)})).
\end{align*}

We consider now the unique homomorphism $\varphi \colon {\m B} \to {\m C}$ such that for every projection $\pi_{b,b'} \colon {\m C} \to {\m C}_{b,b'}$, $\varphi_{b,b'} = \pi_{b,b'}\varphi$. Notice that if $b\neq b'$, then $\varphi_{b,b'}(b) = w(\hat b) \neq w(\hat b') = \varphi_{b,b'}(b')$, because ${\m C}_{b,b'}\not\models_w \hat b\eq\hat b'$. Therefore, $\varphi(b)\neq\varphi(b')$. Hence, $\varphi$ is injective, as we wanted to show.\qedhere
\end{itemproof}
\end{proof}

\subsection{Residual finiteness, free extensions, and the finite embeddability property}

The notion of a free algebra over a partial algebra is a natural generalization of that of a free algebra.

\begin{definition}
A \emph{$\cls K$-free algebra} over a partial $\L$-algebra $\m P$, in symbols  $\F_{\cls K}(\m P)$ or simply $\F(\m P)$, is an algebra ${\m A}\in \cls K$ together with a homomorphism $\eta \colon {\m P} \to {\m A}$ such that for any algebra ${\m B}\in \cls K$ and any homomorphism $\varphi \colon {\m P} \to {\m B}$ there exists a unique homomorphism $\overline\varphi \colon {\m A} \to {\m B}$ with the property that
$\varphi = \overline\varphi\eta$.
\end{definition}

Clearly, a $\cls K$-free algebra over a partial algebra $\m P$ is unique, up to isomorphism, whenever it exists. We prove below that $\F_{\cls V}(\m P)$ exists whenever ${\cls V}$ is a variety, and provide a general method for obtaining it. 

\begin{lemma}\label{lem:k-free algebra}
Let ${\cls V}$ be a variety of $\L$-algebras, $X$ a set of variables, $\m P$ a partial $\L$-algebra, and $\hat\ \colon P \to X$ an injective map. Then the algebra ${\m A} = {\cls V}\pair{\hat P \mid \Diag{\m P}}$, together with the map $\eta \colon {\m P} \to {\m A}$ defined by $\eta(p) = [\hat p]$, is the ${\cls V}$-free algebra over $\m P$.
\end{lemma}

\begin{proof}
We first prove that $\eta$ is a homomorphism. Suppose that $f\in L$ and $p_1,\dots,p_{\tau(f)}\in P$ are such that $f^{\m P}(p_1,\dots,p_{\tau(f)})$ is defined in $\m P$ and such that $f^{\m P}(p_1,\dots,p_{\tau(f)}) = p$. Hence, $f(\hat p_1,\dots,\hat p_{\tau(f)})\eq \hat p$ is in $\Diag{\m P}$, and therefore
\begin{align*}
\eta(f^{\m P}(p_1,\dots,p_{\tau(f)})) &= \eta(p) = [\hat p] = [f(\hat p_1,\dots, \hat p_{\tau(f)})] 
= f^{\m A}([\hat p_1],\dots, [\hat p_{\tau(f)}]) \\
&= f^{\m A}(\eta(p_1),\dots,\eta(p_{\tau(f)})).
\end{align*}

To complete the proof, we show that $\m A$ and $\eta$ satisfy the required universal property. Suppose that ${\m B}\in\cls V$ and $\varphi \colon {\m P} \to {\m B}$ is a homomorphism. Let  $\F(\hat P)$ be the $\cls V$-free algebra over the set $\hat P$, $i\colon P \to \F(\hat P)$ the injective map sending $p\in P$ to $\hat p\in \hat P$,  and $\pi_{\m A} \colon \F(\hat P) \to {\m A}$ be the projection homomorphism.  

Consider the injective map $i\colon P \to \F(\hat P)$, sending $p\in P$ to $\hat p\in \hat P$, and the projection homomorphism $\pi_{\m A} \colon \F(\hat P) \to {\m A}$. Note that $\pi_{\m A}i=\eta$. Since $\F(\hat P)$ is the free algebra over $\hat P$, then there exists a unique homomorphism $\widetilde \varphi \colon \F(\hat P) \to {\m B}$ rendering commutative the exterior part of the diagram: 
\[
\begin{tikzcd}[row sep = small]
{\F(\hat P)}\ar[dr, "\pi_{\m A}"] \ar[rrrdd, bend left, "\widetilde\varphi"] &&& 
\\   {} & {\m A}\ar[drr, dashed, "\overline\varphi"] && 
\\   {\m P}\ar[rrr, "\varphi"] \ar[uu, "i"] \ar[ru, "\eta"] &&& {\m B}
\end{tikzcd}
\]

We next prove that $\Diag{\m P}\subseteq \ker\widetilde\varphi$. Indeed, if $f(\hat p_1,\dots,\hat p_{\tau(f)})\eq \hat p$ is in $\Diag{\m P}$, then $f^{\m P}(p_1,\dots,p_{\tau(f)}) = p$ in $\m P$, and since $\varphi$ is a homomorphism, 
\begin{align*}
\widetilde\varphi(f(\hat p_1,\dots,\hat p_{\tau(f)})) & = f^{\m B}(\widetilde\varphi(\hat p_1),\dots,\widetilde\varphi(\hat p_{\tau(f)}))
 = f^{\m B}(\varphi(p_1),\dots,\varphi(p_{\tau(f)}))
\\  & = \varphi(f^{\m P}(p_1,\dots,p_{\tau(f)})) = \varphi(p) = \widetilde\varphi(\hat p).
\end{align*}

That is, $\pair{ f(\hat p_1,\dots,\hat p_{\tau(f)}), \hat p}\in\ker\widetilde\varphi$, as was to be shown. Hence, there is a unique homomorphism $\overline\varphi \colon {\m A} \to {\m B}$ making the upper triangle of the diagram commutative, that is, $\overline\varphi\pi_{\m A}=\widetilde\varphi$. But then $\varphi=\overline\varphi\eta$. The uniqueness of $\widetilde\varphi$ implies the uniqueness of $\overline\varphi$.
\end{proof}

\begin{proposition}\label{cor:embeddability}
Let ${\cls V}$ be a variety of $\L$-algebras. A partial $\L$-algebra $\m P$ can be embedded into an algebra in ${\cls V}$ if and only if the homomorphism $\eta \colon {\m P} \to \F_{\cls V}(\m P)$ is an embedding.
\end{proposition}

\begin{proof}
One implication is trivial, and the other follows directly from the definition of free algebra over a partial algebra. Indeed, if $\varphi \colon {\m P} \to {\m A}$ is an embedding for some ${\m A}\in{\cls V}$, then the injectivity of $\eta$ follows from the injectivity of $\varphi$ in the diagram below:
\[
\begin{tikzcd}
 {} & {\F_{\cls V}(\m P)} \ar[rd, "\overline\varphi"] &
\\  {\m P} \ar[rr, hook, "\varphi"] \ar[ur, "\eta"] && {\m A}
\end{tikzcd}
\]
This completes the proof of the proposition.
\end{proof}

Before stating the relationship between FEP and residual finiteness, we need a lemma, which may be viewed as a converse of Lemma~\ref{lem:k-free algebra}.

\begin{proposition}\label{prop:finitely:presented:are:Vfree:algebras}
Let ${\cls V}$ be a variety of algebras of finite language and ${\m A}\in{\cls V}$. The following are equivalent:
\begin{enumerate}[(i)]
\item $\m A$ is finitely presentable in ${\cls V}$.
\item $\m A$ is the ${\cls V}$-free algebra over a finite partial algebra.
\end{enumerate}
\end{proposition}

\begin{proof}
By Lemma~\ref{lem:k-free algebra}, every ${\cls V}$-free algebra over partial algebra $\m P$ is of the form $\F_{\cls V}(\m P) = \pair{\hat P\mid \Diag{\m P}}$. Moreover, $\Diag{\m P}$ is finite whenever $\m P$ is finite, and hence $F_{\cls V}(\m P)$ is finitely presentable. 

Conversely, suppose that ${\m A} = \pair{X\mid \Sigma}=\F_{\cls V}(X) / \cg{\F_{\cls V}(X)} (\bar \Sigma)$ is a finitely presented algebra. In view of Corollary~\ref{cor:flat:presentation}, it can be assumed that $\pair{X\mid \Sigma}$ is a flat presentation. Let $\pi_{\m A}\colon \F_{\cls V}(X)\to\m A$ be the associated homomorphism, and let ${\m P} = {\m A}\rest P$ be the full partial subalgebra on a subset  $P$ of $A$ such that $\{[x]\mid x\in X\}\subseteq P$. (Here, and in the remainder of the proof, we write $[x]$ instead of $[x]_{\ker(\pi_{\m A})}$.) We claim that $\m A$, together with the inclusion homomorphism $i \colon {\m P} \to {\m A}$, is the ${\cls V}$-free algebra over $\m P$. Indeed,  let ${\m B}$ be an arbitrary algebra in ${\cls V}$ and $\varphi \colon {\m P} \to {\m B}$ a homomorphism. We need to prove that $\varphi$ can be uniquely extended to a homomorphism $\overline\varphi \colon {\m A} \to {\m B}$. Consider the unique homomorphism $\widetilde\varphi \colon \F(X) \to {\m B}$ such that $\widetilde\varphi(x) = \varphi ([x])$, for every $x\in X$.
\[
\begin{tikzcd}[row sep = small]
{\F(X)}\ar[dr, "\pi_{\m A}"] \ar[rrrdd, bend left, "\widetilde\varphi"] &&& 
\\   {} & {\m A}\ar[drr, dashed, "\overline\varphi"] && 
\\   {\m P}\ar[rrr, "\varphi"]  \ar[ru, "i"] &&& {\m B}
\end{tikzcd}
\]
We prove now that $\ker\pi_{\m A} \subseteq\ker\widetilde\varphi$ or equivalently that $\Sigma\subseteq\ker\widetilde\varphi$. To this end, let $f(x_1,\dots x_{\tau(f)})\eq x$ be an arbitrary equation of $\Sigma$, with $x_1,\dots,x_{\tau(f)},x\in X$. Then, $f^{\m A}([x_1],\dots,[x_{\tau(f)}]) = [x]$. As $[x_1],\dots,[x_{\tau(f)}],[x]\in P$ and $\m P$ is a full partial subalgebra of $\m A$, we also have $f^{\m P}([x_1],\dots,[x_{\tau(f)}]) = [x]$. Hence, using the fact that $\varphi$ and $\widetilde\varphi$ are homomorphisms, we obtain
\begin{align*}
\widetilde\varphi(f(x_1,\dots,x_{\tau(f)})) & = f^{\m B}(\widetilde\varphi(x_1),\dots,\widetilde\varphi(x_{\tau(f)})) = f^{\m B}(\varphi([x_1]),\dots,\varphi([x_{\tau(f)}]))
\\  & = \varphi(f^{\m P}([x_1],\dots,[x_{\tau(f)}])) = \varphi([x]) = \widetilde\varphi(x).
\end{align*}
It follows that $\pair{f(x_1,\dots,x_{\tau(f)}),x}\in\ker\widetilde\varphi$. Thus, $\ker\pi_{\m A} \subseteq\ker\widetilde\varphi$, and therefore there exists a unique homomorphism $\overline\varphi \colon {\m A} \to {\m B}$ rendering commutative the upper triangle of the diagram. As $\varphi\colon\ P\to\m B$ is a homomorphism, a simple inductive argument shows that $\widetilde\varphi(t) = \varphi ([t])$ for all  $[t]\in P$. It follows that $\varphi=\overline\varphi i$. Finally, the uniqueness of  $\overline\varphi$ follows from the fact that $\{[x]\mid x\in X\}$ generates $\m A$.
\end{proof}

\begin{remark}\label{remark:Vfree:over:partial:algebras}
It is important to note that the proof of the previous proposition shows that ${\m A} = \pair{X\mid \Sigma}$ is the ${\cls V}$-free algebra over every finite full partial subalgebra containing $\{[x]_{\ker(\pi_{\m A})}\mid x\in X\}$.
\end{remark}

An algebra $\m A$ is said to be a \emph{subdirect product} of a family of algebras $\{{\m A}_i \mid i\in I\}$ provided there exists an embedding $\varphi \colon {\m A} \to \prod_I{\m A}_i$ such that all homomorphisms $\pi_i\varphi$ are surjective, where  $\pi\colon \m A\to\m A_i$ is the $i$-th projection of the product, for all $i\in I$.

\begin{definition}\label{def:RF}
An algebra $\m A$ in a variety ${\cls V}$ is said to be \emph{residually finite} if it is a subdirect product of a family of finite algebras in ${\cls V}$. 
\end{definition}

\begin{remark}\label{remark:residually:finite}
Note that $\m A$ is residually finite in ${\cls V}$ if and only if for every pair of distinct elements $a,b\in A$, there exist a finite algebra ${\m C}\in{\cls V}$ and a homomorphism $\varphi \colon {\m A} \to {\m C}$ such that $\varphi(a)\neq\varphi(b)$.
\end{remark}

The next proposition is a very straightforward result. For the first part of the proof, we refer the reader to Theorem~10.12 of~\cite{BS81}.

\begin{proposition}\label{prop:FMPandRF}
For a variety ${\cls V}$, the following statements are equivalent:
\begin{enumerate}[(i)]
\item ${\cls V}$ has the FMP.
\item All free algebras in ${\cls V}$ are residually finite.
\item The free algebra $\F_{\cls V}(X)$ over a countable set $X$ is residually finite.
\item All finitely generated free algebras in ${\cls V}$ are residually finite.
\end{enumerate}
\end{proposition}

\begin{proof}\
\begin{itemproof}
\item[(i)${}\Ra{}$(ii):] It is well known that if a variety ${\cls V}$ is generated by a class of algebras $\cls K$, then for every $X$, the free algebra $\F(X)$ is a subdirect product of elements of $\cls K$. Thus, the implication follows from the observation that ${\cls V}$ has the FMP if, and only if, ${\cls V}$ is generated by its finite members.

\item[(ii)${}\Ra{}$(iii):] It is trivial.

\item[(iii)${}\Ra{}$(iv):] It follows from the fact that for any finite set $Y\subseteq X$, $\F_{\cls V}(Y)$ is embeddable in $\F_{\cls V} (X)$.

\item[(iv)${}\Ra{}$(i):] If ${\m A}\not\models t\eq s$, then $\F(Y)\not\models t\eq s$, where $Y$ is the set of the variables of $t$ and $s$. Therefore, there exists a finite ${\m C}\in {\cls V}$ and a homomorphism $\varphi \colon \m F_{\cls V}(Y) \to {\m C}$ such that $\varphi(\bar t)\neq\varphi(\bar s)$, and hence ${\m C}\not\models t\eq s$.\qedhere
\end{itemproof}
\end{proof}

We have already seen that the FEP and the SFMP are equivalent properties for a variety $\cls V$ and imply the FMP. In view of Proposition~\ref{prop:FMPandRF}, the latter property can be characterized in terms of the residual finiteness of the finitely generated free algebras. The next result {due to Evans~\cite{Eva69} (see also~\cite{BaNe72} and~\cite{Eva72})} shows that there is an analogous characterization of the FEP in terms of the residual finiteness of the finitely presentable algebras in $\cls V$.

\begin{theorem}\label{thm:FEP = RF}
Let $\cls V$ be a variety of algebras in a finite language. The following statements are equivalent:
\begin{enumerate}[(i)]
\item $\cls V$ has the FEP.
\item Every finite partial subalgebra of a finitely presentable algebra in $\cls V$ can be embedded into a finite member of $\cls V$.
\item All finitely presentable algebras in $\cls V$ are residually finite.
\end{enumerate}
\end{theorem}

\begin{proof}\
\begin{itemproof}
\item[(i)${}\Ra{}$(ii):] This implication follows by specialization.

\item[(ii)${}\Ra{}$(iii):] Let $\m A$ be a finitely presentable algebra and let $a\not=b\in{\m A}$. In view of Proposition~\ref{prop:finitely:presented:are:Vfree:algebras} and Remark~\ref{remark:Vfree:over:partial:algebras}, there is a finite partial subalgebra $\m P$ of $\m A$ containing $a$ and $b$ such that $\m A$, together with the inclusion $i \colon {\m P} \to {\m A}$, is the ${\cls V}$-free algebra over $\m P$. Condition (ii) implies that there exists an embedding $\varphi \colon {\m P} \to {\m C}$ of $\m P$ into a finite algebra ${\m C}\in{\cls V}$ is finite. As $\m A$ is the ${\cls V}$-free algebra over $\m P$, the homomorphism $\varphi$ can be extended to a homomorphism $\overline\varphi \colon {\m A} \to {\m C}$ such that $\overline\varphi i = \varphi$. Thus, $\overline\varphi(a)\neq\overline\varphi(b)$, and therefore, by Remark~\ref{remark:residually:finite}, $\m A$ is residually finite.

\item[(iii)${}\Ra{}$(i):] Let $\m P$ be a partial subalgebra of an algebra ${\m A}\in{\cls V}$. Then $\m P$ is a partial subalgebra of $\F_{\cls V}(\m P)$, by Proposition~\ref{cor:embeddability}, which is finitely presentable by Proposition~\ref{prop:finitely:presented:are:Vfree:algebras}. Let us call $i$ the inclusion of $\m P$ into $\F_\cls V(\m P)$. For every pair of distinct elements $a,b\in {\m P}$, there exist a finite algebra ${\m C}_{a,b}$ in ${\cls V}$ and a homomorphism $\varphi_{a,b} \colon \F_{\cls V}(\m P) \to {\m C}_{a,b}$ such that $\varphi_{a,b}(a)\neq\varphi_{a,b}(b)$. These homomorphisms induce a homomorphism $\varphi \colon \F_{\cls V}(\m P) \to \prod_{a\neq b}{\m C}_{a,b}$. Clearly, the composition $\varphi i \colon {\m P} \to \prod_{a\neq b}{\m C}_{a,b}$ is an embedding. Moreover, $\prod_{a\neq b}{\m C}_{a,b}$ is a finite product of finite algebras and is therefore finite.\qedhere
\end{itemproof}
\end{proof}

\begin{remark}
Note that both restrictions on the language of ${\cls V}$ to be finite and finitary are necessary in Theorem~\ref{thm:FEP = RF} as it was shown by Banaschewski and Nelson~\cite{BaNe72} and Evans~\cite{Eva72}.
\end{remark}

\subsection{The word problem}

A finitely presented algebra ${\m A} = \pair{X\mid \Sigma}$ in a variety ${\cls V}$ is said to have a \emph{solvable word problem} provided there is an effective procedure for deciding whether the images in $\m A$ of a pair of terms are equal. More precisely, using the notation introduced in the first paragraph of Subsection~\ref{subsection:finitely-presented-algebras}, $\m A$ has a solvable word problem provided there exists an effective procedure for deciding whether 
$\pair{\bar t,\bar s}\in\cg{\F_{\cls V}(X)}(\bar \Sigma)$, for any pair $t, s$ of terms in the variables $X$. The ${\cls V}$ is said to have a \emph{solvable word problem} if every finitely presented algebra in ${\cls V}$ does. {The next theorem is due to Evans~\cite{Eva69}.}

\begin{theorem}
Let ${\cls V}$ be a finitely based variety, that is a variety defined by a finite number of equations. If every finitely presented algebra in ${\cls V}$ is residually finite, then ${\cls V}$ has a solvable word problem.
\end{theorem}

\begin{proof}[Sketch of the proof]
Let us suppose that we are given two terms $s$ and $t$ in $\m T(X)$ and let ${\m A} = {\cls V}\pair{X\mid \Sigma}$ be a finitely presented algebra in ${\cls V}$. We are going to run  two processes in parallel whose combination will provide a positive or a negative answer to the question of whether $s$ and $t$ stand for the same element in $\m A$.

For the first process, we expand the language of ${\cls V}$ by a set of constants $\{c_x \mid x\in X\}$, one new constant for every element of $X$, and, for every term $r\in T(X)$, define the term $\tilde r$ by replacing every variable of $r$ by the corresponding constant. Let ${\cls V}'$ be the variety whose equational basis consists of the equational basis of ${\cls V}$, the equations $\{\tilde r\eq\tilde w \mid r\eq w\in \Sigma\}$, and the equations $\{w\eq c_a \mid a=w^\m A,\ w \text{ a constant in }\L\}$. Consider the equational calculus associated with ${\cls V}'$ (see~\cite{BS81}). This is the consequence relation on the set of equations in the language of ${\cls V}'$ whose axioms are the equations of the basis of ${\cls V}'$, and whose rules state that the relation $\eq$ is a fully invariant congruence. Note that $t$ and $s$ stand for the same element in $\m A$, that is, $\pair{\bar t,\bar s}\in\cg{\F(X)}(\bar \Sigma)$, if and only if $\tilde t\eq\tilde s$ is provable in this calculus. As the set of theorems of the equational calculus is recursively enumerable, a search on the theorems will eventually find $\tilde t \eq\tilde s$,  whenever it is a theorem.

To describe the second process, let us suppose that $t$ and $s$ stand for distinct elements of $\m A$. Then there exists an epimorphism  $\varphi \colon {\m A} \to {\m B}$ from $\m A$ into a finite algebra $\m B\in \cls V$ such that $\varphi([t])\neq \varphi([s])$. This is equivalent to saying that there are a finite algebra ${\m B}$ and an $n$-tuple $b_1,\dots, b_n$ of elements in ${\m B}$ such that for every $r\eq w\in \Sigma$, $r^{\m B}(b_1,\dots,b_n) = w^{\m B}(b_1,\dots,b_n)$ and $t^{\m B}(b_1,\dots,b_n)\neq s^{\m B}(b_1,\dots,b_n)$. The second process runs with the intention of verifying that $s$ and $t$ stand for distinct elements of $\m A$: for $k = 1,2,\dots$ we make a list of all ${\cls V}$-algebras over a fixed set of size $k$. (We can do this because ${\cls V}$ is finitely based.) For each algebra ${\m B}$ on the list, we find all $n$-tuples $b_1,\dots,b_n$ in $B$ such that for every $r\eq w\in \Sigma$, $r^{\m B}(b_1,\dots,b_n) = w^{\m B}(b_1,\dots,b_n)$, and for each such tuple we check whether $s^{\m B}(b_1,\dots,b_n)\neq t^{\m B}(b_1,\dots,b_n)$. If $s$ and $t$ do not determine the same element in $\m A$, such an algebra ${\m B}$ and tuple $b_1,\dots, b_n$ will eventually be found by this process.

Thus, one of the two processes will eventually stop and give a positive or negative answer to the question of whether $s$ is equal to $t$ in $\m A$.
\end{proof}

\begin{remark}
The hypothesis of ${\cls V}$ being a finitely based variety is essential for the last theorem, as Banaschewski and Nelson proved in~\cite{BaNe72}. Likewise the finiteness of the language of ${\cls V}$ is essential, as shown by Evans in~\cite{Eva72}.
\end{remark}

\begin{corollary}
The FEP implies the solvability of the word problem in any finitely based variety on a finite language.
\end{corollary}

\section*{Acknowledgements}
The second and the third named author acknowledge that this research project has received funding from the European Union’s Horizon 2020 research and innovation programme under the Marie Skłodowska-Curie grant agreement No 689176. The fourth-named author acknowledges the support of the National Natural Science Foundation of China under the Grant No.~61473336.



\end{document}